\documentclass[10pt,reqno]{amsart}
\usepackage{amsmath} 
\usepackage{amssymb}
\usepackage{dsfont}
\usepackage[dvips,draft,final]{graphics}
\usepackage{amssymb,amsmath}
\usepackage[T1]{fontenc}
\usepackage{fancyhdr}
\usepackage{url}

\usepackage{color}
\usepackage{graphicx}

\def\squarebox#1{\hbox to #1{\hfill\vbox to #1{\vfill}}}

\newtheorem{Thm}{Theorem}[section]

\newtheorem{lem}{Lemma}[section]

\numberwithin{equation}{section}

\newcommand{\bel}{\begin{equation} \label}
\newcommand{\ee}{\end{equation}}

\newcommand{\pd}{\partial}

\newcommand{\R}{\mathbb{R}}

\def\epsilon{\varepsilon}
\def\phi {\varphi}

\newtheorem{rem}{Remark}[section]
\newtheorem{prop}{Proposition}[section]

\providecommand{\abs}[1]{\left\lvert#1\right\rvert}
\providecommand{\norm}[1]{\left\lVert#1\right\rVert}
\numberwithin{equation}{section}

\renewcommand{\leq}{\leqslant}
\renewcommand{\geq}{\geqslant}
\providecommand{\abs}[1]{\left\lvert#1\right\rvert}
\providecommand{\norm}[1]{\left\lVert#1\right\rVert}

\def\beq{\begin{equation}}
\def\eeq{\end{equation}}
\newcommand{\bea}{\begin{eqnarray}}
\newcommand{\eea}{\end{eqnarray}}
\newcommand{\beas}{\begin{eqnarray*}}
\newcommand{\eeas}{\end{eqnarray*}}

{

\begin{document}

\title[Recovery of  time-dependent damping coefficients appearing in  wave equations]{Recovery of  time-dependent damping coefficients and potentials  appearing in  wave equations from partial data}

\author[Yavar Kian]{Yavar Kian}
\address{Aix Marseille Univ, Universit\'e de Toulon, CNRS, CPT, Marseille, France.}
\email{yavar.kian@univ-amu.fr}
\maketitle

\begin{abstract}
We consider   the inverse problem of  determining   a time-dependent damping coefficient $a$ and a time-dependent potential $q$, appearing in   the wave equation $\partial_t^2u-\Delta_x u+a(t,x)\pd_tu+q(t,x)u=0$ in $Q=(0,T)\times\Omega$, with $T>0$ and $\Omega$  a  $ \mathcal C^2$ bounded domain of $\R^n$, $n\geq2$, from  partial observations of the solutions on $\partial Q$. More precisely, we look for   observations on $\partial Q$ that allow to determine uniquely a large class of  time-dependent damping coefficients $a$ and  time-dependent potentials $q$ without involving an important set of data. Assuming that $a$ is known on $\partial Q$, we prove global unique determination  of  $a\in  W^{1,p}(Q)$, with $p>n+1$, and $q\in L^\infty(Q)$ from partial observations on $\partial Q$. Our problem is related to the determination of  nonlinear terms appearing in  nonlinear wave equations. \\


\end{abstract}

\section{Introduction}
\label{sec-intro}
\setcounter{equation}{0}
\subsection{Statement of the problem }
Let  $\Omega$   be a $\mathcal C^2$ bounded domain  of $\R^n$, $n\geq2$, and  fix  $\Sigma=(0,T)\times\partial\Omega$, $Q=(0,T)\times\Omega$ with $0<T<\infty$.  We consider the wave equation \begin{equation}\label{wave}\partial_t^2u-\Delta_x u+a(t,x)\pd_tu+q(t,x)u=0,\quad (t,x)\in Q,\end{equation}
where  the damping coefficient $a\in L^\infty(Q)$ and the potential $q\in L^\infty(Q)$ are real valued.  In the present paper we
seek unique determination of both $a$ and $q$ from observations  of  solutions of \eqref{wave} on $\partial Q$. 

Let $\nu$ be the outward unit normal vector to $\partial\Omega$, $\partial_\nu=\nu\cdot\nabla_x$ the normal derivative and from now on let $\Box$ and $L_{a,q}$ be  the differential operators $\Box:=\partial_t^2-\Delta_x$, $L_{a,q}:=\Box +a\pd_t+q$. It has been proved by \cite{RS1},  that for $T>\textrm{Diam}(\Omega)$ the data
\begin{equation}\label{data1} \mathcal A_{a,q}=\{(u_{|\Sigma},\partial_\nu u_{|\Sigma}):\ u\in H^1(0,T;L^2(\Omega)),\ \Box u+a\pd_tu+qu=0,\ u_{|t=0}=\partial_tu_{|t=0}=0\}\end{equation}
determines uniquely a time-independent potential $q$ when $a=0$.  The result of \cite{RS1} has been extended to the recovery of a time-independent damping coefficient $a$ by \cite{I1}. Contrary to time-independent coefficients, due to domain of dependence arguments there is  no hope to recover  the restriction of a general time-dependent coefficient  to the set $$D=\{(t,x)\in Q:\  0<t<\textrm{Diam}(\Omega)/2,\  \textrm{dist}(x,\partial\Omega)> t\}$$ from the data $\mathcal A_{a,q}$ (see \cite[Subsection 1.1]{Ki2}). On the other hand, according to \cite[Theorem 4.2]{I}, for $a=0$, the extended set of  data 
\begin{equation}\label{cq}C_{a,q}=\{(u_{\vert\Sigma},u_{\vert t=0}, \partial_tu_{\vert t=0},\partial_\nu u_{\vert\Sigma},u_{\vert t=T},\partial_tu_{\vert t=T}):\  u\in H^1(0,T;L^2(\Omega)),\ L_{a,q}u=0\}\end{equation}
determines uniquely a time-dependent potential $q$. Taking into account the obstruction to the unique determination from the data $\mathcal A_{a,q}$ and the result of \cite{I}, the goal of the present paper is to determine a general time-dependent damping coefficient $a$ and a general time-dependent potential $q$ from partial knowledge of the important set of data  $C_{a,q}$.

\subsection{Physical and mathematical motivations }

In practice, our inverse problem consists of determining  physical properties such as the time evolving damping force and the density of an inhomogeneous medium by probing it with disturbances generated on  the boundary and at  initial time and by measuring  the response  to these disturbances on some parts of the boundary and at the end of the experiment. The goal is to determine the functions $a$ and $q$ which measure the damping force and the property of the medium. The determination of such a time-dependent coefficients can also correspond to the recovery of some time evolving  properties that can not be modeled by  time-independent coefficients.

As mentioned in \cite{Ki2,Ki3}, following the strategy set in \cite{I2} for parabolic equations, the recovery  
of   nonlinear terms, appearing in some suitable nonlinear wave equations, from observations  on $\partial Q$  can be reduced to the determination of  time-dependent coefficients, with weak regularity,  appearing in a linear wave equation. In this context, the regularity of the time-dependent coefficients depends on the regularity of the solutions of the nonlinear equations. Thus, for this application of our problem it is important to weaken as much as   possible the regularity of the admissible time-dependent coefficients.

\subsection{State of the art }

The determination of  coefficients for hyperbolic equations  from boundary measurements  has  attracted many attention in recent years.  Many authors  considered the recovery of time-independent potentials from  observations given by the set $\mathcal A_{a,q}$ defined by \eqref{data1} for $a=0$. In \cite{RS1}, the authors proved that, for $a=0$,   $\mathcal A_{a,q}$ determines uniquely a time-independent  potential $q$.  The uniqueness by partial boundary observations has been considered in \cite{E1}. We also precise that the  stability issue for this problem has been studied by \cite{BD,BJY,Ki,Mo,SU,SU2}. 

Some authors treated also the recovery of both time-independent damping coefficients and potentials from boundary measurements. In \cite{I1}, Isakov extended the result of  \cite{RS1}, to the recovery of both  damping coefficients and potentials from the data $\mathcal A_{a,q}$. For $n=3$, \cite{IS} proved stable recovery of the restriction of both time-independent damping coefficients and potentials on the intersection of the domain and a half-space from measurements on the intersection of the boundary of the domain and the same half-space.  Following the strategy set by \cite{BK}, \cite{BCIY,LT1,LT2} proved uniqueness and stability in the recovery of both  damping coefficients and potentials from a single boundary measurements. In some recent work, 
\cite{AC} proved a log-type stability estimate in the recovery of  time-independent damping coefficients and  potentials appearing  in a dissipative wave equation from the initial boundary map.

All the above mentioned results are concerned  with time-independent coefficients. Several authors considered the problem of determining time-dependent coefficients for hyperbolic equations. In \cite{St}, Stefanov proved the recovery of a time-dependent potential appearing in the wave equation  from the knowledge of scattering data by using some  properties of the light-ray transform.   In \cite{RS}, Ramm and  Sj\"ostrand considered the determination of a time-dependent potential $q$ from the data $(u_{|\R\times\partial\Omega}, \partial_\nu u_{|\R\times\partial\Omega})$ of forward solutions of \eqref{wave} with $a=0$ on the infinite time-space cylindrical domain $\R_t\times\Omega$  instead of $Q$ ($t\in\R$ instead of $0<t<T<\infty$).   Rakesh and  Ramm \cite{RR} treated  this  problem at finite time on $Q$, with $T>\textrm{Diam} (\Omega)$, and they determined uniquely $q$ restricted to some  subset of $Q$ from $\mathcal A_{a,q}$ with $a=0$.  Isakov established in \cite[Theorem 4.2]{I} unique determination of  general time-dependent potentials on the whole domain $Q$ from the extended data $C_{a,q}$ given by \eqref{cq} with $a=0$. Using  a result of unique continuation borrowed from \cite{T}, Eskin \cite{E2} proved unique recovery of time-dependent coefficients  analytic with respect to the time variable $t$ from  partial knowledge of the data $\mathcal A_{a,q}$.  Salazar \cite{S} extended the result of \cite{RS} to more general coefficients. Moreover, \cite{W} stated stability in the recovery of  X-ray transforms of time-dependent potentials on a  manifold and \cite{A} proved log-type stability in the determination of time-dependent potentials from the data considered by \cite{I} and \cite{RR}. We mention also the recent work of \cite{BB} where the authors have extended the results of 
\cite{A} to the recovery of both time-dependent damping coefficients and potentials. In \cite{Ki2,Ki3}, the author considered both uniqueness and stability in the recovery of some general time dependent potential $q$ from (roughly speaking) half of the data considered by \cite{I}. To our best knowledge the results of \cite{Ki2,Ki3} are stated with the weakest conditions so far that guaranty determination of a general time dependent potential, appearing in a wave equation, at finite time. We also mention that \cite{Ch,CK,CKS,FK,GK} examined the determination of time-dependent coefficients for fractional diffusion, parabolic and Schr\"odinger equations and  proved stability estimate for these problems.

\subsection{Main result}
To state our main result, we  first introduce some intermediate tools and notations. 
For all $\omega\in\mathbb S^{n-1}=\{x\in\R^n:\ \abs{x}=1\}$ we introduce the $\omega$-shadowed and $\omega$-illuminated faces
\[\partial\Omega_{+,\omega}=\{x\in\partial\Omega:\ \nu(x)\cdot\omega>0\},\quad \partial\Omega_{-,\omega}=\{x\in\partial\Omega:\ \nu(x)\cdot\omega\leq0\}\]
of $\partial\Omega$. Here, for all $k\in\mathbb N^*$, $\cdot$ corresponds to the scalar product in $\R^k$ defined by
\[ x\cdot y=x_1y_1+\ldots +x_ky_k,\quad x=(x_1,\ldots,x_k)\in \R^k,\ y=(y_1,\ldots,y_k)\in \R^k.\]
We define also the parts of the lateral boundary $\Sigma$ given by 
$\Sigma_{\pm,\omega}=(0,T)\times \partial\Omega_{\pm,\omega}$.
We fix $\omega_0\in \mathbb S^{n-1}$ and we consider $V=(0,T)\times V'$  with $V'$  a closed  neighborhood of  $\partial\Omega_{-,\omega_0}$ in $\partial\Omega$.

 The main purpose of this paper is to prove  the unique global  determination of  time-dependent  and real valued damping coefficient $a\in L^\infty(Q)$ and  $q\in L^\infty(Q)$  from the data
\[C_{a,q}^*=\{(u_{|\Sigma},u_{|t=0},\partial_t u_{|t=0},\partial_\nu u_{|V}, u_{|t=T}):\ u\in H^1(0,T; L^2(\Omega)),\ L_{a,q}u=0\}.\]
We refer to Section 2 for the definition of this set. Our main result can be  stated as follows.

\begin{Thm}\label{thm1} 
Let $q_1,\ q_2 \in L^\infty(Q)$ and let  $a_1,a_2\in  W^{1,p}(Q)$ with $p>n+1$. Assume that 
\begin{equation}\label{thm1bb} a_1(t,x)=a_2(t,x),\quad (t,x)\in\pd Q.\end{equation}
Then, the condition
\begin{equation}\label{thm1a}C_{a_1,q_1}^*=C_{a_2,q_2}^*\end{equation}
implies that $a_1=a_2$ and $q_1=q_2$.
\end{Thm}

To our best knowledge this paper is the first treating uniqueness in the recovery of time-dependent damping coefficients. Moreover, it seems that with \cite{E2,E3,S} this paper is the first considering recovery of time-dependent coefficients of order one and it appears that this  work is the first  treating this problem for general coefficients at finite time (\cite{E2,E3}   proved recovery of coefficients analytic with respect to the time variable $t$, \cite{S} considered the problem for all time $t\in\mathbb R$). We point out that our uniqueness result is stated for general coefficients with observations close to the one considered by \cite{Ki2,Ki3}, where recovery of time-dependent potentials is proved with conditions that seems to be one of the weakest so far for a general class of time-dependent coefficients. Indeed, the only difference between \cite{Ki2,Ki3} and the present paper comes from the restriction on the Dirichlet boundary condition  and the initial value (\cite{Ki2,Ki3} consider Dirichlet boundary condition supported on a neighborhood of the $\omega_0$-shadowed face  and vanishing at $t=0$, where here we do not make restriction on the support of the Dirichlet boundary condition and at $t=0$).
We also mention that in contrast to \cite{E2}, we do not apply results of unique continuation  that require the analyticity with respect to the time variable $t$ of the coefficients.

Let us  observe that even for $T$ large,  according to the obstruction to uniqueness given by domain of dependence arguments (see \cite[Subsection 1.1]{Ki2}), there is no hope to remove all the information on $\{t=0\}$ and  $\{t=T\}$  for the global recovery of general time-dependent coefficients. Therefore, for our problem the  extra information on $\{t=0\}$ and  $\{t=T\}$, of solutions $u$ of \eqref{wave}, can not be completely removed.

The main tools in our analysis are Carleman estimates with linear weight and   geometric optics (GO in short) solutions suitably designed for our inverse problem. In a similar way to \cite{BJY, Ki2,Ki3},  we use  GO solutions taking the form of exponentially growing and exponentially decaying solutions   in accordance with our  Carleman estimate in order to both recover the coefficients and restrict the observations. Our GO solutions differ from the one of \cite{E1,E2,I1,RS1,RS,S} and, combined with our Carleman estimate, they  make it possible to prove  global recovery of time-dependent coefficients from partial knowledge of the set $C_{a,q}$ without using additional smoothness or geometrical assumptions. Even if this strategy is inspired by \cite{BJY, Ki2,Ki3} (see also \cite{BU,KSU} for the original idea in the case of elliptic equations), due to the presence of a variable coefficient of order one in \eqref{wave}, our approach differs from \cite{BJY, Ki2,Ki3} in many aspects. Indeed, to prove our Carleman estimate we perturb the linear weight and we prove this estimate by using a convexity argument that allows us to absorb the damping coefficient. Moreover, in contrast to \cite{Ki2,Ki3} our GO are designed for the recovery of the damping coefficient and we can not construct them by applying properties of solutions of PDEs with constant coefficients. We remedy to this by considering  Carleman estimates in Sobolev space of negative  order  and by using these estimates to build our GO solutions.  This construction  is inspired by the one used in \cite{DKSU,KSU} for the recovery of Schr\"odinger operators from partial boundary measurements.

Note that condition \eqref{thm1bb} is meaningful for damping coefficients that actually depend on the time variable $t$ ($\pd_ta_j\neq0$, $j=1,2$). Indeed, for time-independent  damping coefficients $a_1$, $a_2$,  \eqref{thm1bb} implies that $a_1=a_2$. However, by modifying the form of the principal part of the GO given in Section 4 in accordance with \cite{I1}, for $T>\textrm{Diam}(\Omega)$ we believe that we can restrict condition \eqref{thm1bb} to the knowledge of time-independent  damping coefficients on $\pd \Omega$ 
($a_1=a_2$ on $\pd\Omega$ instead of \eqref{thm1bb}). In order to avoid the inadequate expense of the size of the paper we will not treat that case.

We believe that, with some suitable modifications, the approach developed in the present paper can be used for proving recovery of more general time-dependent coefficients of order one including a magnetic field associated to a time-dependent magnetic potential.

\subsection{Outline}

This paper is organized as follows. In Section 2 we  introduce some tools and we define  the set of data $C_{a,q}^*$. In Section 3, we prove our first Carleman estimate which will play an important role in our analysis. In Section 4 we extend and mollify the damping coefficient and we  introduce the principal part of our GO solutions. In Section 5, we derive  a Carleman estimate in  Sobolev space of negative order. Then, using this estimate, we build suitable GO solutions associated to \eqref{wave}.  Finally  in Section 6, we combine the GO solutions of Section 5 with the Carleman estimate of Section 3 to prove Theorem \ref{thm1}. 

\section{Preliminary results}
In the present section we define the set of data $C_{a,q}^*$ and we recall some properties of the solutions of \eqref{wave} for any $a,q\in L^\infty(Q)$.  For this purpose, in a similar way to  \cite{Ki2}, we will introduce some preliminary tools. We define the  space
\[H_{\Box}(Q)=\{u\in H^1(0,T;L^2(\Omega)):\ \Box u=(\partial_t^2-\Delta_x) u\in L^2(Q)\},\]
 with the norm
\[\norm{u}^2_{H_{\Box}(Q)}=\norm{u}_{H^1(0,T;L^2(\Omega))}^2+\norm{(\partial_t^2-\Delta_x) u}_{L^2(Q)}^2.\]
We consider also  the  space
\[S=\{u\in H^1(0,T;L^2(\Omega)):\ (\partial_t^2-\Delta_x)u=0\}\]
and topologize it as a closed subset of $H^1(0,T;L^2(\Omega))$. Indeed, let $(f_k)_{k\in\mathbb N}$ be a sequence lying in $S$ that converge to $f$ in $H^1(0,T;L^2(\Omega))$. Then, $(f_k)_{k\in\mathbb N}$
converge to $f$ in the sense of $D'(Q)$ and in the same way $((\partial_t^2-\Delta_x)f_k)_{k\in\mathbb N}$ converge to $(\partial_t^2-\Delta_x)f$ in the sense of $D'(Q)$. Now using the fact that for all $k\in\mathbb N$, $(\partial_t^2-\Delta_x)f_k=0$  we deduce that $(\partial_t^2-\Delta_x)f=0$. This proves that $f\in S$ and that $S$ is a closed subspace of $H^1(0,T;L^2(\Omega))$.

In view of  \cite[Proposition 4]{Ki2},  the maps
\[\tau_0w=(w_{\vert\Sigma},w_{\vert t=0},\partial_t w_{\vert t=0}) ,\quad \tau_1w=(\partial_\nu w_{\vert\Sigma},w_{\vert t=T},\partial_t w_{\vert t=T}), \quad w\in \mathcal C^\infty(\overline{Q}),\]
can be extended continuously to $\tau_0:H_{\Box}(Q)\rightarrow H^{-3}(0,T; H^{-\frac{1}{2}}(\partial\Omega))\times H^{-2}(\Omega)\times H^{-4}(\Omega)$, $\tau_1:H_{\Box}(Q)\rightarrow H^{-3}(0,T; H^{-\frac{3}{2}}(\partial\Omega))\times H^{-2}(\Omega)\times H^{-4}(\Omega)$. Here for all $ w\in \mathcal C^\infty(\overline{Q})$ we set
\[\tau_0w=(\tau_{0,1}w,\tau_{0,2}w,\tau_{0,3}w),\quad \tau_1w=(\tau_{1,1}w,\tau_{1,2}w,\tau_{1,3}w),\]
where
\[ \tau_{0,1}w=w_{\vert\Sigma},\    \tau_{0,2}w=w_{\vert t=0},\ \tau_{0,3}w=\partial_tw_{\vert t=0},\  \tau_{1,1}w=\partial_\nu w_{\vert\Sigma},\ \tau_{1,2}w=w_{\vert t=T},\ \tau_{1,3}w=\partial_tw_{\vert t=T}.\]
Therefore, we can introduce
\[\mathcal H=\{\tau_0u:\ u\in H_{\Box}(Q)\}\subset  H^{-3}(0,T; H^{-\frac{1}{2}}(\partial\Omega))\times H^{-2}(\Omega)\times H^{-4}(\Omega).\]
Following \cite{Ki2} (see also \cite{BU,CKS1,CKS2,NS}  in the case of elliptic equations), in order to define an appropriate topology on $\mathcal H$ we consider the restriction of $\tau_0$ to the space $S$. Indeed, by repeating the arguments used in \cite[Proposition 1]{Ki2}, one can check that  the restriction of $\tau_0$ to $S$ is one to one and onto. Thus, we can   use  $\tau_0^{-1}:=({\tau_0}_{|S})^{-1}$ to define the norm of $\mathcal H$ by
\[\norm{(f,v_0,v_1)}_{\mathcal H}=\norm{\tau_0^{-1}(f,v_0,v_1)}_{H^1(0,T;L^2(\Omega))},\quad (f,v_0,v_1)\in\mathcal H,\]
with $\tau_0$ considered as its restriction to $S$. 
Let us introduce  the IBVP
\begin{equation}\label{eq1}\left\{\begin{array}{ll}\partial_t^2u-\Delta_x u+a(t,x)\pd_tu+q(t,x)u=0,\quad &\textrm{in}\ Q,\\  u(0,\cdot)=v_0,\quad \partial_tu(0,\cdot)=v_1,\quad &\textrm{in}\ \Omega,\\ u=g,\quad &\textrm{on}\ \Sigma.\end{array}\right.\end{equation}
We are now in position to state existence and uniqueness of solutions of this IBVP for $(g,v_0,v_1)\in \mathcal H$.
\begin{prop}\label{p6} Let $(g,v_0,v_1)\in \mathcal H$, $a\in L^\infty(Q)$ and $q\in L^\infty(Q)$. Then, the IBVP \eqref{eq1} admits a unique weak solution $u\in H^1(0,T;L^2(\Omega))$ satisfying 
\bel{p6a}
\norm{u}_{H^1(0,T;L^2(\Omega))}\leq C\norm{(g,v_0,v_1)}_{\mathcal H}
\ee
and the boundary operator $B_{a,q}: (g,v_0,v_1)\mapsto (\tau_{1,1}u_{|V}, \tau_{1,2}u)$ is a bounded operator from $\mathcal H$ to\\
 $H^{-3}(0,T; H^{-\frac{3}{2}}(V'))\times H^{-2}(\Omega)$.
\end{prop}
 \begin{proof} We split $u$ into two terms $u=v+\tau_0^{-1}(g,v_0,v_1)$ where $v$ solves
\bel{Eq1}\left\{ \begin{array}{rcll} \partial_t^2v-\Delta_x v+a\pd_tv+qv& = & (-a\pd_t-q)\tau_0^{-1}(g,v_0,v_1), & (t,x) \in Q ,\\ 

v_{\vert t=0}=\partial_t v_{\vert t=0}&=&0,&\\
 v_{\vert\Sigma}& = & 0.& \end{array}\right.
\ee
Since $(-a\pd_t-q)\tau_0^{-1}(g,v_0,v_1)\in L^2(Q)$, from  the theory developed in \cite[Chapter 3, Section 8]{LM1}, one can check that the IBVP \eqref{Eq1} admits a unique solution $v\in \mathcal C^1([0,T];L^2(\Omega))\cap \mathcal C([0,T];H^1_0(\Omega))$ satisfying
\bel{p6b}
\begin{aligned}\norm{v}_{\mathcal C^1([0,T];L^2(\Omega))}+\norm{v}_{\mathcal C([0,T];H^1_0(\Omega))} &\leq C\norm{(-a\pd_t-q)\tau_0^{-1}(g,v_0,v_1)}_{L^2(Q)}\\
\ &\leq C(\norm{q}_{L^\infty(Q)}+\norm{a}_{L^\infty(Q)})\norm{\tau_0^{-1}(g,v_0,v_1)}_{H^1(0,T;L^2(\Omega))}.\end{aligned}\ee
Therefore, $u=v+\tau_0^{-1}(g,v_0,v_1)$ is the unique solution of \eqref{eq1} and estimate \eqref{p6b} implies \eqref{p6a}. Now let us show the last part of the proposition. For this purpose fix $(g,v_0,v_1)\in\mathcal H$ and consider $u\in H^1(0,T;L^2(\Omega))$ the solution of \eqref{eq1}. Note first that  $(\partial_t^2-\Delta_x) u=-a\pd_tu-qu\in L^2(Q)$. Thus,  $u\in H_{\Box}(Q)$ and  $\tau_{1,1}u\in H^{-3}(0,T; H^{-\frac{3}{2}}(\partial\Omega))$, $\tau_{1,2}u \in H^{-2}(\Omega)$ with
\[\begin{aligned}\norm{\tau_{1,1}u}^2+\norm{\tau_{1,2}u}^2\leq C^2\norm{u}^2_{H_{\Box}(Q)}&=C^2(\norm{u}^2_{H^1(0,T;L^2(\Omega))}+\norm{a\pd_tu+qu}^2_{L^2(Q)})\\
\ &\leq C^2(1+2\norm{a}^2_{L^\infty(Q)}+2\norm{q}^2_{L^\infty(Q)})\norm{u}_{H^1(0,T;L^2(\Omega))}^2.\end{aligned}\]
Combining this with \eqref{p6a} we deduce that $B_{a,q}$ is a bounded operator from $\mathcal H$ to $H^{-3}(0,T; H^{-\frac{3}{2}}(V'))\times H^{-2}(\Omega)$.\end{proof}

From now on we consider the set $C_{a,q}^*$ to be the graph of the boundary operator $B_{a,q}$ given by
\[C_{a,q}^*=\{(g,v_0,v_1,B_{a,q}(g,v_0,v_1)):\ (g,v_0,v_1)\in \mathcal H\}.\]

\section{Carleman estimates}
This section will be devoted to the proof of a Carleman estimate with linear weight associated with \eqref{wave} which will be one of the main tools in our analysis. More precisely, we will consider the following.

\begin{Thm}\label{c1}  Let $\omega\in \mathbb S^{n-1}$, $a,\ q\in L^\infty(Q)$ and  $u\in\mathcal C^2(\overline{Q})$.  If $u$ satisfies the condition 
 \begin{equation}\label{ttc1}u_{\vert \Sigma}=0,\quad u_{\vert t=0}=\partial_tu_{\vert t=0}=0,\end{equation}
then there exists $\lambda_1>1$ depending only on  $\Omega$, $T$ and $M\geq \norm{q}_{L^\infty(Q)}+\norm{a}_{L^\infty(Q)}$ such that the estimate
\begin{equation}\label{c1a}\begin{array}{l}\lambda \int_\Omega e^{-2\lambda( T+\omega\cdot x)}\abs{\partial_tu_{\vert t=T}}^2dx+\lambda\int_{\Sigma_{+,\omega}}e^{-2\lambda(t+\omega\cdot x)}\abs{\partial_\nu u}^2\abs{\omega\cdot\nu(x) } d\sigma(x)dt
+\lambda^2\int_Qe^{-2\lambda(t+\omega\cdot x)}\abs{u}^2dxdt\\
+\int_Qe^{-2\lambda(t+\omega\cdot x)}(\abs{\nabla u}^2+\abs{\pd_tu}^2)dxdt
\leq C\left(\int_Qe^{-2\lambda(t+\omega\cdot x)}\abs{L_{a,q}u}^2dxdt+ \lambda^3\int_\Omega e^{-2\lambda(T+\omega\cdot x)}\abs{u_{\vert t=T}}^2dx\right)\\
\ \ \ +C\left(\lambda\int_\Omega e^{-2\lambda(T+\omega\cdot x)}\abs{\nabla_xu_{\vert t=T}}^2dx+\lambda\int_{\Sigma_{-,\omega}}e^{-2\lambda(t+\omega\cdot x)}\abs{\partial_\nu u}^2\abs{\omega\cdot\nu(x) }d\sigma(x)dt\right)\end{array}\end{equation}
holds true for $\lambda\geq \lambda_1$  with $C$  depending only on  $\Omega$, $T$ and $M\geq \norm{q}_{L^\infty(Q)}+\norm{a}_{L^\infty(Q)}$.
\end{Thm}
 For $a=0$, the Carleman estimate \eqref{c1a} has already been established in \cite[Theorem 2]{Ki2} by applying some results of \cite{BJY}. In contrast to the equation without the damping coefficient ($a=0$),  due to the presence of a variable coefficient of order one, we can not  derive \eqref{c1a} from \cite[Theorem 2]{Ki2}. In order to establish this Carleman estimate, in a similar way to \cite{DKSU, KSU}, we need to perturb our linear weight in order to absorb  the damping coefficient. More precisely, we introduce a new parameter $s$ independent of $\lambda$ that will be precised later, we consider, for $\lambda>s>1$, the perturbed weight
\bel{phi}\phi_{\pm\lambda,s}(t,x):=\pm\lambda (t+\omega\cdot x)-s{t^2\over 2}\ee
and we define
\[P_{a,q,\lambda,s}:=e^{-\phi_{\lambda,s}}L_{a,q}e^{\phi_{\lambda,s}}.\]
\ \\
\textbf{Proof of Theorem \ref{c1}.} We fix $v=e^{-\phi_{\lambda,s}}u$ such that $$\int_Qe^{-2\phi_{\lambda,s}}|L_{a,q}u|^2dxdt=\int_Q|P_{a,q,\lambda,s}v|^2dxdt.$$
In the remaining part of this proof we  will systematically omit the subscripts $\lambda,s$ in $\phi_{\lambda,s}$ and without lost of generality we will assume that $u$ is real valued. Our first goal will be to establish for $\lambda$   sufficiently large and for suitable fixed value of $s$ depending on $\Omega$, $T$ and $\norm{a}_{L^\infty(Q)}+\norm{q}_{L^\infty(Q)}$, the estimate
\bel{7}\begin{aligned}\norm{P_{a,q,\lambda,s}v}^2_{L^2(Q)}\geq&{\lambda\over 4}\int_\Omega |\pd_tv(T,x)|^2dx+2\int_Q(|\pd_tv|^2+|\nabla_x v|^2)dxdt-7\lambda \int_\Omega |\nabla_xv(T,x)|^2dx\\
\ &+\lambda\int_\Sigma |\pd_\nu v|^2\omega\cdot\nu d\sigma(x)dt+5\lambda^2\int_Q|v|^2dxdt-6Ts\lambda^3\int_\Omega |v(T,x)|^2dx.\end{aligned}\ee
Then,  we will deduce \eqref{c1a}. We decompose $P_{a,q,\lambda,s}$ into three terms
\[P_{a,q,\lambda,s}=P_1+P_2+P_3,\]
with
\[P_1=\pd_t^2-\Delta_x+\Box \phi+|\pd_t\phi|^2-|\nabla_x\phi|^2,\quad P_2=2\pd_t\phi\pd_t-2\nabla_x\phi\cdot\nabla_x,\quad P_3=a\pd_t+a\pd_t\phi+q.\]
 Recall that
\[\pd_t\phi=\lambda-st,\quad \pd_t^2\phi=-s,\quad \nabla_x\phi=\lambda\omega,\]

\[\begin{aligned}P_1vP_2v=&2\pd_t^2v\pd_t\phi \pd_tv-2\pd_t^2v\nabla_x\phi\cdot\nabla_xv-2\Delta_xv\pd_t\phi \pd_tv+2\Delta_xv\nabla_x\phi\cdot\nabla_xv\\
\ &+(\Box \phi_{\lambda,s}+|\pd_t\phi|^2-|\nabla_x\phi|^2) vP_2v.\end{aligned}\]
We have
\[2\int_Q\pd_t^2v\pd_t\phi \pd_tvdxdt=\int_Q \pd_t\phi\pd_t|\pd_tv|^2dxdt=\int_\Omega \pd_t\phi(T,x)|\pd_tv(T,x)|^2dx-\int_Q\pd_t^2\phi|\pd_tv|^2dxdt.\]
For $\lambda>2sT$ we obtain 
\bel{1}2\int_Q\pd_t^2v\pd_t\phi \pd_tvdxdt\geq {\lambda\over 2}\int_\Omega |\pd_tv(T,x)|^2dx+s\int_Q|\pd_tv|^2dxdt.\ee
We have also
\[-2\int_Q\pd_t^2v\nabla_x\phi\cdot\nabla_xvdxdt=-2\lambda\int_\omega \pd_t v(T)\omega\cdot\nabla_xv(T)dx+\lambda\int_Q\omega\cdot\nabla_x(|\pd_tv|^2)dxdt\]
and using the fact that $v_{|\Sigma}=0$ we get
\[\int_Q\omega\cdot\nabla_x(|\pd_tv|^2)dxdt=\int_Q\textrm{div}_x(|\pd_tv|^2 \omega)dxdt=0,\]
\bel{2}-2\int_Q\pd_t^2v\nabla_x\phi\cdot\nabla_xvdxdt=-2\lambda\int_\Omega \pd_t v(T)\omega\cdot\nabla_xv(T)dx.\ee
In a similar way, we find
\[-2\int_Q\Delta_xv\pd_t\phi \pd_tvdx dt=\int_Q\pd_t\phi \pd_t|\nabla v|^2dx dt\]
and using the formula
\[\int_Q\pd_t\phi \pd_t|\nabla v|^2dx dt=\int_\Omega\pd_t\phi(T) |\nabla v|^2(T)dx dt -\int_Q\pd_t^2\phi |\nabla v|^2dx dt,\]
we obtain
\bel{3}-2\int_Q\Delta_xv\pd_t\phi \pd_tvdx dt=s\int_Q |\nabla v|^2dx dt+\int_\Omega\pd_t\phi(T) |\nabla v|^2(T)dx\geq  s\int_Q |\nabla v|^2dx dt+{\lambda\over 2}\int_\Omega |\nabla v|^2(T)dx.\ee
Moreover, in a similar way to \cite[Lemma 2.1]{BU}, we get
\bel{4}2\int_Q\Delta_xv\nabla_x\phi\cdot\nabla_xvdxdt=2\int_Q\Delta_xv\omega\cdot\nabla_xvdxdt=\lambda\int_\Sigma |\pd_\nu v|^2\omega\cdot\nu d\sigma(x)dt.\ee
Now note that
\[\Box \phi+|\pd_t\phi|^2-|\nabla_x\phi|^2=-s+(\lambda-st)^2-\lambda^2=s^2t^2-2\lambda st-s.\]
Then, we have
\[\begin{array}{l}\int_Q(\Box \phi_{\lambda,s}+|\pd_t\phi|^2-|\nabla_x\phi|^2) vP_2v\\
=\int_Q(s^2t^2-2\lambda st-s)\pd_t\phi\pd_t|v|^2dxdt-\int_Q(s^2t^2-2\lambda st-s)\textrm{div}(|v|^2\omega )dxdt\\
=\int_Q(-s^3t^3+3\lambda s^2t^2+(s^2-2\lambda^2s)t-s\lambda)\pd_t|v|^2dxdt\\
=-\int_Q[-3s^3t^2+6\lambda s^2t+(s^2-2\lambda^2s)]|v|^2dxdt+\int_\Omega(-s^3T^3+3\lambda s^2T^2+(s^2-2\lambda^2s)T-s\lambda)|v(T,x)|^2dx.\end{array}\]
Choosing $\lambda$ such that
\[\lambda>\max\left(s+6\lambda sT,s^2T^2+{\lambda\over T}\right)^{1/2},\]
we get
\bel{5}\int_Q(\Box \phi_{\lambda,s}+|\pd_t\phi|^2-|\nabla_x\phi|^2) vP_2v\geq \lambda^2s\int_Q|v|^2dxdt-3T\lambda^2s\int_\Omega |v(T,x)|^2dx.\ee
Combining estimates \eqref{1}-\eqref{5}, we obtain 
\bel{6}\begin{aligned}\norm{P_1v+P_2v}^2_{L^2(Q)}\geq2\int_QP_1vP_2v\geq&{\lambda\over 2}\int_\Omega |\pd_tv(T,x)|^2dx+2s\int_Q(|\pd_tv|^2+|\nabla_x v|^2)dxdt-7\lambda \int_\Omega |\nabla_xv(T,x)|^2dx\\
\ &+2\lambda\int_\Sigma |\pd_\nu v|^2\omega\cdot\nu d\sigma(x)dt+2s\lambda^2\int_Q|v|^2dxdt-6Ts\lambda^2\int_\Omega |v(T,x)|^2dx.\end{aligned}\ee
On the other hand, we have

\[\begin{aligned}\norm{P_{a,q,\lambda,s}v}^2_{L^2(Q)}&\geq {\norm{P_1v+P_2v}^2_{L^2(Q)}\over 2}-\norm{P_3v}_{L^2(Q)}^2\\
\ &\geq {\norm{P_1v+P_2v}^2_{L^2(Q)}\over 2}-3\norm{a}_{L^\infty(Q)}^2\int_Q|\pd_t v|^2dxdt-3\left(\norm{a}_{L^\infty(Q)}^2\lambda +\norm{q}_{L^\infty(Q)}^2\right)\int_Q| v|^2dxdt.\end{aligned}\]
Fixing $s=3\norm{a}_{L^\infty(Q)}^2+6$ and $\lambda >3\norm{q}_{L^\infty(Q)}^2+\max\left(s^2+6\lambda sT,4s^2T^2+{\lambda\over T}\right)^{1/2}$, we deduce \eqref{7} from \eqref{6}.
Armed with \eqref{7}, we will complete the proof of the Carleman estimate \eqref{c1a}. Note that, condition \eqref{ttc1} implies $\partial_\nu v_{\vert\Sigma}=e^{-\lambda(t+\omega\cdot x)}e^{{st^2\over 2}}\partial_\nu u_{\vert\Sigma}$ and we deduce that
\bel{8}\lambda\int_\Sigma |\pd_\nu v|^2\omega\cdot\nu d\sigma(x)dt\geq \lambda\int_{\Sigma_+} e^{-2\lambda(t+\omega\cdot x)}|\pd_\nu u|^2\omega\cdot\nu d\sigma(x)dt+e^{sT^2}\lambda\int_{\Sigma_-} e^{-2\lambda(t+\omega\cdot x)}|\pd_\nu u|^2\omega\cdot\nu d\sigma(x)dt.\ee
Moreover, using the fact that
\[\partial_tu=\partial_t(e^{\phi} v)=(\lambda-st) u+e^{\lambda(t+\omega\cdot x)}e^{-{st^2\over 2}} \partial_tv,\quad \nabla_x v=e^{-\phi}(\nabla_x u-\lambda u\omega),\]
we obtain
\[\int_Qe^{-2\lambda(t+\omega\cdot x)}(|\pd_tu|^2+|\nabla_x u|^2)dxdt\leq 4\lambda^2\int_Q|v|^2dxdt+2\int_Q(|\pd_tv|^2+|\nabla_x v|^2)dxdt,\]
\[\int_\Omega e^{-2\lambda(T+\omega\cdot x)}\abs{\partial_tu(T,x)}^2dx\leq 2\int_\Omega \abs{\partial_tv(T,x)}^2dx+2\lambda^2\int_\Omega e^{-2\lambda(T+\omega\cdot x)}\abs{u(T,x)}^2dx,\]
\[\int_\Omega \abs{\nabla_x v(T,x)}^2dx\leq 2\lambda^2e^{sT^2}\int_\Omega e^{-2\lambda(T+\omega\cdot x)}\abs{u(T,x)}^2dx+2e^{sT^2}\int_\Omega e^{-2\lambda(T+\omega\cdot x)}\abs{\nabla_x u(T,x)}^2dx.\]
Combining these estimates with \eqref{7}, \eqref{8}, we get
\bel{9}\begin{array}{l}\int_Qe^{-2\lambda(t+\omega\cdot x)}(|\pd_tu|^2+|\nabla_x u|^2)dxdt+\lambda^2\int_Qe^{-2\lambda(t+\omega\cdot x)}|u|^2dxdt\\
\ \\
+{\lambda\over 8}\int_\Omega e^{-2\lambda(T+\omega\cdot x)}\abs{\partial_tu(T,x)}^2dx
+\lambda\int_{\Sigma_+} e^{-2\lambda(t+\omega\cdot x)}|\pd_\nu u|^2\omega\cdot\nu d\sigma(x)dt\\
\ \\
\leq (6Ts+14e^{sT^2}+1)\lambda^3\int_\Omega e^{-2\lambda(T+\omega\cdot x)}\abs{u(T,x)}^2dx+14\lambda e^{sT^2}\int_\Omega e^{-2\lambda(T+\omega\cdot x)}\abs{\nabla_x u(T,x)}^2dx\\
\ \\
+e^{sT^2}\lambda\int_{\Sigma_-} e^{-2\lambda(t+\omega\cdot x)}|\pd_\nu u|^2|\omega\cdot\nu| d\sigma(x)dt+e^{sT^2}\int_Qe^{-2\lambda(t+\omega\cdot x)}|L_{a,q}u|^2dxdt.\end{array}\ee
From this last estimate we deduce \eqref{c1a}. 
\qed
\begin{rem}\label{rr} By density, estimate \eqref{c1a} can be extended to any function $u\in\mathcal C^1([0,T]; L^2(\Omega))\cap \mathcal C([0,T]; H^1(\Omega))$ satisfying \eqref{ttc1}, $(\partial_t^2-\Delta_x)u\in L^2(Q)$ and $\partial_\nu u\in L^2(\Sigma)$.\end{rem} 
Now that our Carleman estimate is proved we will extend it into  Sobolev space of negative order and apply it to construct our GO solutions. But first let us  define the principal part of our GO solutions that will allow us to recover  the damping coefficient. For this purpose we need some suitable approximation of our damping coefficients.
\section{Approximation of the damping coefficients}

Let us first remark that the GO solutions that we use for the  recovery of the damping coefficients should depend explicitly on the damping coefficients. In order to avoid additional smoothness assumptions on the class of admissible coefficients, in a similar way to \cite{Ki4,Sa1}, we consider 	GO solutions depending on some smooth approximation of the damping coefficients instead of the damping coefficients themselves. The main purpose of this section is to define our choice for the smooth approximation of the damping coefficients and to introduce the part of our GO solutions that will be used for the recovery of the damping coefficient.  From now on we fix the coefficients $a_1,a_2\in W^{1,p}(Q)$, with $p>n+1$, and $q_1,q_2\in L^\infty(Q)$. Moreover, we assume that
\bel{conda} a_1(t,x)=a_2(t,x),\quad (t,x)\in \pd Q.\ee
For all $r>0$ we define $B_r:=\{(t,x)\in\R^{1+n}:\ |(t,x)|<r\}$. Then, according to \cite[Theorem 5, page 181]{[St]}, fixing $R>0$ such that $\overline{Q}\subset B_R$, we can introduce $\tilde{a}_1\in W^{1,p}(\R^{1+n})$, with supp$\tilde{a}_1\subset B_R$ such that
$$\tilde{a}_1(t,x)=a_1(t,x),\quad (t,x)\in Q.$$
We define also $\tilde{a}_2$ by 
$$\tilde{a}_2(t,x)=\left\{\begin{array}{l} a_2(t,x),\quad (t,x)\in Q\\ \tilde{a}_1(t,x),\quad (t,x)\in \R^{1+n}\setminus Q.\end{array}\right.$$
Then, in view of \eqref{conda}, we have $\tilde{a}_2\in W^{1,p}(\R^{1+n})$ and for $a=a_2-a_1$ extended by zero outside $Q$ we have
\bel{condb}\tilde{a}_2-\tilde{a}_1=a.\ee
Now let us fix $\chi\in\mathcal C^\infty_0(\R^{1+n})$ such that $\chi\geq0$, $\int_{\R^{1+n}}\chi(t,x)dtdx=1$, supp$\chi\subset B_1$, and let us define $\chi_\lambda$ by
$\chi_\lambda(t,x)=\lambda^{{n+1\over 3}}\chi(\lambda^{{1\over3}}t,\lambda^{{1\over3}}x)$
and, for $j=1,2$, we fix
$$a_{j,\lambda}(t,x):=\chi_\lambda*\tilde{a}_j(t,x)=\int_{\R^{1+n}}\chi_\lambda(t-\tau,x-y)\tilde{a}_j(\tau,y)dyd\tau.$$
For $j=1,2$, since $\tilde{a}_j\in W^{1,p}(\R^{1+n})$ with $p>n+1$, by the Sobolev embedding theorem (e.g.  \cite[Theorem 1.4.4.1]{Gr}) we have $\tilde{a}_j\in \mathcal C^{\alpha}(\R^{1+n})$ with $\alpha=1-{n+1\over p}$ and $\mathcal C^{\alpha}(\R^{1+n})$ the set of $\alpha$-H\"olderian functions. Thus, one can check that
\bel{a1a}\norm{a_{j,\lambda}-\tilde{a}_j}_{L^\infty(\R^{1+n})}\leq C\lambda^{-{\alpha \over3}},\ee
\bel{a1b}\norm{a_{j,\lambda}}_{W^{k,\infty}(\R^{1+n})}\leq C_k\lambda ^{{k\over 3}},\quad \norm{a_{j,\lambda}}_{H^k(\R^{1+n})}\leq C_k\lambda ^{{k-1\over 3}}\quad k\geq2,\ee
with $C$ and $C_k$ independent of $\lambda$. In view of \eqref{condb},  we remark that $a_{\lambda}:=\chi_\lambda*a=a_{2,\lambda}-a_{1,\lambda}$.
Then, for $\zeta\in (1,-\omega)^\bot:=\{(t,x)\in\R\times\R^{n}:\ t-\omega\cdot x=0\}$, we fix 
\bel{conde1} b_{1,\lambda}(t,x)=e^{-i\zeta\cdot(t,x)}\exp\left(-{\int_0^{+\infty} a_{1,\lambda}((t,x)+s(1,-\omega))ds\over 2}\right),\ee
\bel{conde2} b_{2,\lambda}(t,x)=\exp\left({\int_0^{+\infty} a_{2,\lambda}((t,x)+s(1,-\omega))ds\over 2}\right).\ee
According to \eqref{a1a}-\eqref{a1b} and to the fact that, for $j=1,2$, supp$a_{j,\lambda}\subset B_{R+1}$, we have
\bel{condf}\norm{b_{j,\lambda}}_{W^{k,\infty}(\R^{1+n})}\leq C\lambda ^{{k\over3}},\quad \norm{b_{j,\lambda}}_{H^k(Q)}\leq C\lambda ^{{k-1\over3}},\quad k\geq1\ee
and
\bel{condg} \norm{(2\pd_t-2\omega\cdot\nabla_x-a_1)b_{1,\lambda}}_{L^2(Q)}\leq C\lambda^{-{\alpha\over3}},\ \norm{(2\pd_t-2\omega\cdot\nabla_x+a_2)b_{2,\lambda}}_{L^2(Q)}\leq C\lambda^{-{\alpha\over3}} .\ee
Note that  $b_{1,\lambda}$ and $b_{2,\lambda}$ are respectively some smooth approximation of the functions
$$ b_{1}(t,x)=e^{-i\zeta\cdot(t,x)}e^{-{\int_0^{+\infty} \tilde{a}_{1}((t,x)+s(1,-\omega))ds\over 2}},\quad b_{2}(t,x)=e^{{\int_0^{+\infty} \tilde{a}_{2}((t,x)+s(1,-\omega))ds\over 2}},$$
which are respectively a solution of the transport equations
$$2\pd_tb_1-2\omega\cdot\nabla_xb_1-a_1b_1=0,\quad 2\pd_tb_2-2\omega\cdot\nabla_xb_2+a_2b_2=0,\quad \textrm{in } Q.$$
Recall that here $\tilde{a}_j$, $j=1,2$, is the extension of the damping coefficients $a_j$ defined at the beginning of this section. By replacing in  our GO solutions the functions $b_1$, $b_2$, whose regularity depends on the one  of the coefficients $a_1$ and $a_2$, with their approximation $b_{1,\lambda}$, $b_{2,\lambda}$,  we can weaken the regularity  assumption imposed on admissible damping coefficients from $W^{2,\infty}(Q)$ to $W^{1,p}(Q)$. This approach requires also less information about the damping coefficients on $\partial Q$. More precisely, if in our construction  we use the expression $b_j$ instead of  $b_{j,\lambda}$, $j=1,2$,  then, following our strategy, we can prove Theorem \ref{thm1}  only for damping coefficients $a_1,a_2\in W^{2,\infty}(Q)$  satisfying
$$\partial_t^k\partial_x^\alpha a_1(t,x)=\partial_t^k\partial_x^\alpha a_2(t,x),\quad (t,x)\in\partial Q,\ k\in\mathbb N,\ \alpha\in\mathbb N^n,\ k+|\alpha|\leq1,$$
which is  a more restrictive condition than \eqref{thm1bb}.
We mention also that this approach has already been considered by some authors in the context of recovery of non-smooth first order coefficients appearing in elliptic equations (see for instance \cite{Sa1}).

From now on by  using the tools introduced in this section, we will construct two kind of solutions of  equations of the from \eqref{wave}. We consider first solutions $u_1\in H^1(Q)$  of $L_{-a_1,q_1}u_1=0$ taking the form
\[u_1(t,x)=e^{-\lambda(t+x\cdot\omega)}(b_{1,\lambda}(t,x)+R_1(t,x))\]
and solutions $u_2\in H^1(Q)$  of the equation $$L_{a_2,q_2}u_2=0,\ \textrm{in }Q,\quad  u_2(0,x)=0,\ x\in\Omega,$$  taking the form
\[u_2(t,x)=e^{\lambda(t+x\cdot\omega)}(b_{2,\lambda}(t,x)+R_2(t,x)).\]
Here $R_j$, $j=1,2$, denotes the remainder term in the expression of the solution $u_j$ with respect to the parameter $\lambda$ in such a way that there exists $\mu\in (0,1)$ such that  $\lambda\norm{R_j}_{L^2(Q)}+ \norm{R_j}_{H^1(Q)}\leq C\lambda^{\mu}$.

\section{Goemtric optics solutions }
In this section we consider exponentially decaying solutions $u_1\in H^1(Q)$ of the equation $(\pd_t^2-\Delta_x-a_1\pd_t +q_1)u_1=0$ in $Q$ taking the form
\bel{GO1} u_1(t,x)=e^{-\lambda(t+x\cdot\omega)}(b_{1,\lambda}(t,x)+w_1(t,x)),\ee
and exponentially growing solution $u_2\in H^1(Q)$ of the equation $(\pd_t^2-\Delta_x+a_2\pd_t +q_2)u_2=0$ in $Q$ taking the form
\bel{GO2} u_2(t,x)=e^{\lambda (t+x\cdot\omega)}(b_{2,\lambda}(t,x)+w_2(t,x))\ee
where  $\lambda>1$, $\omega\in\mathbb S^{n-1}:=\{y\in\R^n:\ |y|=1\}$ and the term $w_j\in H^1(Q)$, $j=1,2$, satisfies
\bel{CGO11}\norm{w_j}_{H^1(Q)}+\lambda\norm{w_j}_{L^2(Q)}\leq C\lambda^{{3-\alpha\over3}},\ee
with $C$ independent of  $\lambda$.
We summarize these results in the following way.
\begin{prop}\label{p1} There exists $\lambda_2>\lambda_1$ such that for $\lambda>\lambda_2$ we can find a solution $u_1\in H^1(Q)$ of $L_{-a_1,q_1}u_1=0$ in $Q$ taking the form \eqref{GO1} 
with $w_1\in H^1(Q)$ satisfying \eqref{CGO11} for $j=1$.\end{prop} 
\begin{prop}\label{p3} There exists $\lambda_3>\lambda_2$ such that for $\lambda>\lambda_3$ we can find a solution  $u_2\in H^1(Q)$ of $L_{a_2,q_2}u_2=0$ in $Q$ taking the form \eqref{GO2}
with $w_2\in H^1(Q)$ satisfying \eqref{CGO11} for $j=2$.\end{prop} 
\subsection{Exponentially decaying solutions}
This subsection is devoted to the proof of Proposition \ref{p1}. To construct the   exponentially decaying solutions $u_1\in H^1(Q)$ of the form \eqref{GO1}  we first introduce some preliminary tools and a suitable Carleman estimate in Sobolev space of negative order. More precisely, before starting the proof of Proposition \ref{p1}, we introduce some preliminary notations and we consider three intermediate results including a new Carleman estimate. In a similar way to \cite{BJY}, for all $m\in\R$, we introduce the space $H^m_\lambda(\R^{1+n})$ defined by
\[H^m_\lambda(\R^{1+n})=\{u\in\mathcal S'(\R^{1+n}):\ (|(\tau,\xi)|^2+\lambda^2)^{m\over 2}\hat{u}\in L^2(\R^{1+n})\},\]
with the norm
\[\norm{u}_{H^m_\lambda(\R^{1+n})}^2=\int_\R\int_{\R^n}(|(\tau,\xi)|^2+\lambda^2)^{m}|\hat{u}(\tau,\xi)|^2 d\xi d\tau.\]
Here for all tempered distribution $u\in \mathcal S'(\R^{1+n})$, we denote by $\hat{u}$ the Fourier transform of $u$ which, for $u\in L^1(\R^{1+n})$, is defined by
$$\hat{u}(\tau,\xi):=\mathcal Fu(\tau,\xi):= (2\pi)^{-{n+1\over2}}\int_{\R^{1+n}}e^{-it\tau-ix\cdot \xi}u(t,x)dtdx.$$
From now on, for $m\in\R$, $\tau\in\R$ and $\xi\in \R^n$,  we set $$\left\langle (\tau,\xi),\lambda\right\rangle=(\tau^2+|\xi|^2+\lambda^2)^{1\over2}$$
and $\left\langle D,\lambda\right\rangle^m u$ defined by
\[\left\langle D,\lambda\right\rangle^m u=\mathcal F^{-1}(\left\langle (\tau,\xi),\lambda\right\rangle^m \mathcal Fu).\]
For $m\in\R$ we define also the class of symbols
\[S^m_\lambda=\{c_\lambda\in\mathcal C^\infty(\R^{1+n}\times\R^{1+n}):\ |\pd_t^j\pd_x^\alpha\pd_\tau^k\pd_\xi^\beta c_\lambda(t,x,\tau,\xi)|\leq C_{j,\alpha,k,\beta}\left\langle (\tau,\xi),\lambda\right\rangle^{m-k-|\beta|},\ j,k\in\mathbb N,\ \alpha,\beta\in\mathbb N^n\}.\]
Following \cite[Theorem 18.1.6]{Ho3}, for any $m\in\R$ and $c_\lambda\in S^m_\lambda$, we define $c_\lambda(t,x,D_t,D_x)$, with $D_t=-i\pd_t$, $D_x=-i\nabla_x$, by
\[c_\lambda(t,x,D_t,D_x)u(t,x)=(2\pi)^{-{n+1\over 2}}\int_{\R^{1+n}}c_\lambda(t,x,\tau,\xi)\hat{u}(\tau,\xi)e^{it\tau+ix\cdot \xi}d\tau d\xi.\]
For all $m\in\R$, we set also $OpS^m_\lambda:=\{c_\lambda(t,x,D_t,D_x):\ c_\lambda\in S^m_\lambda\}$.
We fix
$$P_{a,\omega,\pm\lambda}:=e^{\mp \lambda(t+x\cdot\omega)}(L_{a,q}-q)e^{\pm \lambda(t+x\cdot\omega)}$$
and we consider the following Carleman estimate.

\begin{lem}\label{l1} Let $a\in W^{1,p}(Q)$. Then, there exists $\lambda_2'>\lambda_1$ such that  
\bel{car2}\norm{v}_{L^2(\R^{1+n})}\leq C\norm{P_{a,\omega,\lambda}v}_{H^{-1}_\lambda(\R^{1+n})},\quad v\in\mathcal C^\infty_0(Q),\ \ \lambda>\lambda_2',\ee
with $C$ independent of $v$ and $\lambda$.
\end{lem}
\begin{proof} 
For $\phi_{\lambda,s}$ given by \eqref{phi}, we  consider
\[P_{a,  \lambda,s}=e^{-\phi_{\lambda,s}}(L_{a,q}-q)e^{\phi_{\lambda,s}}\]
and in a similar way to Theorem \ref{c1} we decompose $P_{a,\lambda,s}$ into three terms
\[P_{a, \lambda,s}=P_{1}+P_{2}+P_{a,3},\]
with
\[P_{1}=\pd_t^2-\Delta_x +s^2t^2-2\lambda ts-s,\quad P_{2}=2(\lambda-st)\pd_t-2\omega\cdot\nabla_x,\quad P_{a,3}=a\pd_t+a(\lambda-ts).\]
Fixing $\tilde{Q}=(-1,T+1)\times \tilde{\Omega}$ with $\tilde{\Omega}$ a bounded open set of $\R^n$ such that $\overline{\Omega}\subset\tilde{\Omega}$ , we extend the function  $a$  to $\R^{1+n}$ with  $ a\in W^{1,p}(\R^{1+n})$ and $a$ supported on $\tilde{Q}$. We first prove the Carleman estimate
\bel{car}\norm{v}_{L^2(\R^{1+n})}\leq C\norm{P_{a,\lambda,s}v}_{H^{-1}_\lambda(\R^{1+n})},\quad v\in\mathcal C^\infty_0(Q).\ee
For this purpose, we fix $w\in \mathcal C^\infty_0(\tilde{Q})$ and we  consider the quantity
\[\left\langle D,\lambda\right\rangle^{-1}(P_{1}+P_{2})\left\langle D,\lambda\right\rangle w.\]
From now on and in all the remaining parts of this proof $C>0$ denotes a generic constant independent of $s$, $w$ and $\lambda$.
Applying the properties of composition of pseudoddifferential operators (e.g. \cite[Theorem 18.1.8]{Ho3}), we find
\bel{t2c}\left\langle D,\lambda\right\rangle^{-1}(P_{1}+P_{2})\left\langle D,\lambda\right\rangle=P_{1}+P_{2}+R_\lambda(t,D_t,D_x),\ee
where $R_\lambda$ is defined by
\[R_\lambda(t,\tau,\xi)=\pd_\tau\left\langle (\tau,\xi),\lambda\right\rangle^{-1}D_t(p_{1}(t,\tau,\xi)+p_{2}(t,\tau,\xi))\left\langle (\tau,\xi),\lambda\right\rangle+\underset{\left\langle (\tau,\xi),\lambda\right\rangle\to+\infty}{\mathcal O}(1),\]
with
$$p_{1}(t,\tau,\xi)=-\tau^2+|\xi|^2-2s\lambda t+s^2t^2-s ,\quad p_{2}(t,\tau,\xi)=2i((\lambda-st)\tau-\omega\cdot\xi).$$
Therefore, we have
$$R_\lambda(t,\tau,\xi)={i\tau(-2is\tau-2\lambda s+2s^2t)\over \tau^2+|\xi|^2+\lambda^2}+\underset{\left\langle (\tau,\xi),\lambda\right\rangle\to+\infty}{\mathcal O}(1)$$
and it follows
\bel{t2d} \norm{R_\lambda(t,D_t,D_x)w}_{L^2(\R^n)}\leq Cs^2\norm{w}_{L^2(\R^{1+n})}.\ee
On the other hand,  applying \eqref{7} to $w$ with $Q$ replaced by  $\tilde{Q}$,  we get
\[\norm{P_{1}w+P_{2}w}_{L^2(\R^{1+n})}\geq C\left[s^{1/2}\left(\norm{\nabla w}_{L^2(\R^{1+n})}+\norm{\pd_t w}_{L^2(\R^{1+n})}\right)+s^{1/2}\lambda\norm{ w}_{L^2(\R^{1+n})}\right].\]
Combining this estimate with \eqref{t2c}-\eqref{t2d}, for ${\lambda\over s^2}$ sufficiently large, we obtain
 \bel{t2e}\begin{array}{l} \norm{(P_{1}+P_{2})\left\langle D,\lambda\right\rangle w}_{H^{-1}_\lambda(\R^{1+n})}\\
=\norm{\left\langle D,\lambda\right\rangle^{-1}(P_{1}+P_{2})\left\langle D,\lambda\right\rangle w}_{L^2(\R^n)}\\ \geq  C\left[s^{1/2}\left(\norm{\nabla w}_{L^2(\R^{1+n})}+\norm{\pd_t w}_{L^2(\R^{1+n})}\right)+s^{1/2}\lambda\norm{ w}_{L^2(\R^{1+n})}-s^2\norm{w}_{L^2(\R^{1+n})}\right]\\
\geq Cs^{1/2}\left(\norm{\nabla w}_{L^2(\R^{1+n})}+\norm{\pd_t w}_{L^2(\R^{1+n})}+\lambda\norm{ w}_{L^2(\R^{1+n})}\right).\end{array}\ee
Moreover, we have
\[\norm{P_{a,3}\left\langle D,\lambda\right\rangle w}_{H^{-1}_\lambda(\R^{1+n})}\leq \norm{a\pd_t\left\langle D,\lambda\right\rangle w}_{H^{-1}_\lambda(\R^{1+n})}+\norm{a(\lambda-st)\left\langle D,\lambda\right\rangle w}_{H^{-1}_\lambda(\R^{1+n})}.\]
Now let us consider the following result.

\begin{lem}\label{l7} Let $a\in W^{1,p}(\R^n)$ be supported on $\tilde{Q}$. Then, for all $f\in L^2(\R^{1+n})$ we have
\bel{l7a}\norm{af}_{H^{-1}_\lambda(\R^{1+n})}\leq C\norm{a}_{W^{1,p}(\R^{1+n})}\norm{f}_{H^{-1}_\lambda(\R^{1+n})},\ee
with $C$ a constant depending only on $\tilde{Q}$.\end{lem}
The proof of this lemma will be postponed to the end of the present demonstration. Applying estimate \eqref{l7a} with $\lambda>s(T+1)$,  we obtain
\[\begin{array}{l}\norm{P_{a,3}\left\langle D,\lambda\right\rangle w}_{H^{-1}_\lambda(\R^{1+n})}\\
\leq C\left(\norm{a}_{W^{1,p}(\R^{1+n})}\norm{\pd_t\left\langle D,\lambda\right\rangle w}_{H^{-1}_\lambda(\R^{1+n})}+\norm{(\lambda-st)a}_{W^{1,p}(\R^{1+n})} \norm{\left\langle D,\lambda\right\rangle w}_{H^{-1}_\lambda(\R^{1+n})}\right)\\
\leq C\left(\norm{a}_{W^{1,p}(\R^{1+n})}\norm{\pd_t w}_{L^2(\R^{1+n})}+\norm{a}_{W^{1,p}(\R^{1+n})}\lambda \norm{ w}_{L^2(\R^{1+n})}\right),\end{array}\]
with $C$ depending only on $\Omega$ and $T$.
Thus, choosing $s>C\norm{a}_{W^{1,p}(\R^{1+n})}+1$, \eqref{t2e} implies
\bel{t2f}\norm{P_{a,\lambda,s}\left\langle D,\lambda\right\rangle w}_{H^{-1}_\lambda(\R^{1+n})}\geq C\norm{w}_{H^1_\lambda(\R^{1+n})}.\ee
Now we fix $\psi_0\in\mathcal C^\infty_0(\tilde{Q})$ satisfying $\psi_0=1$ on $\overline{Q_1}$, with $Q_1=(-\delta,T+\delta)\times \Omega_1$, $\delta\in(0,1)$ and $\Omega_1$ an open neighborhood of $\overline{\Omega}$ such that $\overline{\Omega_1}\subset\tilde{\Omega}$. Then, we fix $w=\psi_0 \left\langle D,\lambda\right\rangle^{-1} v$ and for $\psi_1\in\mathcal C^\infty_0(Q_1)$ satisfying $\psi_1=1$ on $Q$, we get $(1-\psi_0 )\left\langle D,\lambda\right\rangle^{-1} v=(1-\psi_0 )\left\langle D,\lambda\right\rangle^{-1}\psi_1 v$. According to \cite[Theorem 18.1.8]{Ho3}, we have $(1-\psi_0) \left\langle D,\lambda\right\rangle\psi_1\in OpS^{-\infty}_\lambda$ and it follows
\[ \begin{aligned}\norm{v}_{L^2(\R^{1+n})}&=\norm{\left\langle D,\lambda\right\rangle^{-1} v}_{H^1_\lambda(\R^{1+n})}\\
\ &\leq \norm{w}_{H^1_\lambda(\R^{1+n})}+\norm{(1-\psi_0)\left\langle D,\lambda\right\rangle^{-1}\psi_1 v}_{L^2(\R^{1+n})}\\
\ &\leq \norm{w}_{H^1_\lambda(\R^{1+n})}+{C\norm{v}_{L^2(\R^{1+n})}\over\lambda^2} .\end{aligned}\]
In the same way, we find
$$\begin{aligned}\norm{P_{a,\lambda,s} v}_{H^{-1}_\lambda(\R^{1+n})}&\geq \norm{P_{a,\lambda,s}\left\langle D,\lambda\right\rangle w}_{H^{-1}_\lambda(\R^{1+n})}-\norm{P_{a,\lambda,s}\left\langle D,\lambda\right\rangle (1-\psi_0 )\left\langle D,\lambda\right\rangle^{-1}\psi_1 v}_{H^{-1}_\lambda(\R^{1+n})}\\
\ &\geq \norm{P_{a,\lambda,s}\left\langle D,\lambda\right\rangle w}_{H^{-1}_\lambda(\R^{1+n})}-{C\norm{v}_{L^2(\R^{1+n})}\over\lambda^2}.\end{aligned}$$

Combining these estimates with \eqref{t2f}, we deduce that \eqref{car} holds true for a sufficiently large value of  $\lambda$. Applying \eqref{car} for a fixed value of $s$, we deduce that there exists $\lambda_2'>0$ such that \eqref{car2} is fulfilled.\end{proof}
Now that the proof of Lemma \ref{l1} is completed let us consider Lemma \ref{l7}.

\textbf{Proof of Lemma \ref{l7}.} We fix $h\in H^{1}(\R^{1+n})$. Since $a\in W^{1,p}(\R^{1+n})$ is compactly supported, we have both $a\in L^\infty(\R^{1+n})$, $a\in W^{1,n+1}(\R^{1+n})$ and combining the Sobolev embedding theorem with the H\"older inequality one can check that $ah\in H^{1}(\R^{1+n})$ with
$$\norm{ah}_{H^1(\R^{1+n})}\leq C\norm{a}_{W^{1,p}(\R^{1+n})}\norm{h}_{H^1(\R^{1+n})},$$
with $C$ depending only on $\tilde{Q}$. Therefore, we find
$$\begin{aligned}\norm{ah}_{H^1_\lambda (\R^{1+n})}^2&= \lambda^2\norm{ah}_{L^2(\R^{1+n})}^2+\norm{\nabla_{t,x}(ah)}_{L^2(\R^{1+n})}^2\\
\ &\leq \lambda^2\norm{a}_{L^\infty(\R^{1+n})}^2\norm{h}_{L^2(\R^{1+n})}^2+C^2\norm{a}_{W^{1,p}(\R^{1+n})}^2\norm{h}_{H^1(\R^{1+n})}^2\\
\ &\leq (\norm{a}_{L^\infty(\R^{1+n})}^2+C^2\norm{a}_{W^{1,p}(\R^{1+n})}^2)[\lambda^2\norm{h}_{L^2(\R^{1+n})}^2+\norm{h}_{H^1(\R^{1+n})}^2]\end{aligned}$$
and from this estimate we obtain
\bel{l7b}\norm{ah}_{H^1_\lambda (\R^{1+n})}\leq C\norm{a}_{W^{1,p}(\R^{1+n})}\norm{h}_{H^1_\lambda (\R^{1+n})}.\ee
On the other hand, we have
$$\abs{\left\langle af,h\right\rangle_{H^{-1}_\lambda(\R^{1+n}),H^{1}_\lambda(\R^{1+n})}}=\abs{\left\langle f,ah\right\rangle_{L^2(\R^{1+n})}}=\abs{\left\langle f,ah\right\rangle_{H^{-1}_\lambda(\R^{1+n}),H^{1}_\lambda(\R^{1+n})}}\leq \norm{f}_{H^{-1}_\lambda(\R^{1+n})}\norm{ah}_{H^1_\lambda (\R^{1+n})}$$
and \eqref{l7b} implies
$$\abs{\left\langle af,h\right\rangle_{H^{-1}_\lambda(\R^{1+n}),H^{1}_\lambda(\R^{1+n})}}\leq C\norm{a}_{W^{1,p}(\R^{1+n})}\norm{f}_{H^{-1}_\lambda(\R^{1+n})}\norm{h}_{H^1_\lambda (\R^{1+n})}.$$
From this estimate we deduce \eqref{l7a}.\qed\\

Using the Carleman estimate \eqref{car2}, we can consider solutions $v\in H^1(Q)$ of the equation $P_{-a,\omega,-\lambda}v=F$ with $F\in L^2(Q)$ in the following way.
\begin{lem}\label{l1bis} Let $\omega\in \mathbb S^{n-1}$, $\lambda>\lambda_2'$,  $a\in W^{1,p}(Q)$. Then, there exists a bounded operator $E_{a,\lambda}\in \mathcal B(L^2(Q);H^1(Q))$ such that:
\bel{l1bisa} P_{-a,\omega,-\lambda}E_{a,\lambda} F=F,\quad F\in L^2(Q),\ee
\bel{l1bisb}\lambda\norm{E_{a,\lambda} F}_{L^2(Q)}+\norm{E_{a,\lambda} F}_{H^1(Q)}\leq C\norm{F}_{L^2(Q)},\ee
where $C>0$ is independent of $\lambda$ and $F$.\end{lem}
\begin{proof}
 We will use estimate \eqref{car2} to construct a solution $v\in H^1(\R^{1+n})$ of the equation $P_{-a,\omega,-\lambda}v=F$ in $Q$.
Applying the Carleman estimate \eqref{car2}, we define the linear form $\mathcal K$ on $\{P_{a,\omega,\lambda}z:\ z\in\mathcal C^\infty_0(Q)\}$, considered as a subspace of $H^{-1}_\lambda(\R^{1+n})$ by
\[\mathcal K(P_{a,\omega,\lambda}z)=\int_QzFdxdt,\quad z\in\mathcal C^\infty_0(Q).\]
Then, \eqref{car2} implies 
\[|\mathcal K(P_{a,\omega,\lambda}z)|\leq C\norm{F}_{L^2(Q)}\norm{P_{a,\omega,\lambda}z}_{H^{-1}_\lambda(\R^{1+n})},\quad z\in\mathcal C^\infty_0(Q).\]
Thus, by the Hahn Banach theorem we can extend $\mathcal K$ to a continuous linear form on $H^{-1}_\lambda(\R^{1+n})$ still denoted by $\mathcal K$ and satisfying $\norm{\mathcal K}\leq C\norm{F}_{L^2(Q)}$. Therefore, there exists $v\in H^{1}_\lambda(\R^{1+n})$ such that 
\[\left\langle h,v\right\rangle_{H^{-1}_\lambda(\R^{1+n}),H^{1}_\lambda(\R^{1+n})}=\mathcal K(h),\quad h\in H^{-1}_\lambda(\R^{1+n}).\]
Choosing $h=P_{a,\omega,\lambda}z$ with $z\in \mathcal C^\infty_0(Q)$ proves that $v$ satisfies $P_{-a,\omega,-\lambda}v=F$ in $Q$.
Moreover, we have $\norm{v}_{H^{1}_\lambda(\R^{1+n})}\leq \norm{\mathcal K}\leq C\norm{F}_{L^2(Q)}$. Therefore, fixing $E_{a,\lambda} F=v_{|Q}$ we deduce easily \eqref{l1bisa}-\eqref{l1bisb}.\end{proof}

Armed with Lemma \ref{l1}-\ref{l1bis}, we are now in position to complete the proof of Proposition \ref{p1}.
\ \\
 \textbf{Proof of proposition \ref{p1}.} Note first that the condition $L_{-a_1,q_1}u_1=0$ is fulfilled if and only if $w$
solves
\bel{CGO12} P_{-a_1,\omega,-\lambda}w=-q_1w-P_{-a_1,\omega,-\lambda}b_{1,\lambda}-q_1b_{1,\lambda}=-q_1w+\lambda (2\pd_tb_{1,\lambda}-2\omega\cdot\nabla_xb_{1,\lambda}-a_1b_{1,\lambda})-L_{-a_1,q_1}b_{1,\lambda}.\ee
Therefore, fixing 
$$G=-\lambda (2\pd_tb_{1,\lambda}-2\omega\cdot\nabla_xb_{1,\lambda}-a_1b_{1,\lambda})+L_{-a_1,q_1}b_{1,\lambda}$$
we can consider $w$ as a solution of $w=-E_{a_1,\lambda} (q_1w+G)$. We will solve this equation by considering the fixed point of the map
\[\begin{array}{rccl} \mathcal G: & L^2(Q) & \to & L^2(Q), \\
 \ \\ & F & \mapsto &-E_{a_1,\lambda}\left[q_1F+G\right]. \end{array}\]
For this purpose, we recall that \eqref{condf}-\eqref{condg} imply that $\norm{G}_{L^2(Q)}\leq C\lambda^{{3-\alpha\over3}}$ and \eqref{l1bisb} implies
\bel{eqeq}\norm{E_{a_1,\lambda}G}_{L^2(Q)}\leq C\lambda^{-1}\norm{G}_{L^2(Q)}\leq C\lambda^{-{\alpha\over3}},\ee
with $C$ independent of $\lambda$. Thus, fixing $M_1>0$, we can find $\lambda_2>\lambda_2'$ such that for $\lambda\geq \lambda_2$ the map $\mathcal G$ admits a unique fixed point $w$ on $\{u\in L^2(Q): \norm{u}_{L^2(Q)}\leq M_1\}$. In addition, we have $w\in H^1(Q)$ and from condition \eqref{l1bisb}, \eqref{eqeq} we obtain
\[\norm{w}_{H^1(Q)}+\lambda\norm{w}_{L^2(Q)}\leq  C\norm{G}_{L^2(Q)}\leq C\lambda^{{3-\alpha\over3}}.\]
From this estimate we deduce \eqref{CGO11}. This completes the proof of Proposition \ref{p1}.\qed

\subsection{Exponentially growing solutions}
In this subsection we will consider the construction of the exponentially growing solutions given by Proposition \ref{p3}.  For this purpose we consider the operator
\[P_{-a,-\lambda,s}=e^{-\phi_{-\lambda,s}}(\pd_t^2-a\pd_t-\Delta_x)e^{\phi_{-\lambda,s}},\]
with $\phi_{-\lambda,s}$ given by \eqref{phi}. By extending $a$ to $a\in W^{1,p}(\R^{1+n})$ with $a$ supported on $\tilde{Q}$, we assume that $P_{-a,\omega,-\lambda}$ and $P_{-a,-\lambda,s}$ are  differential operators acting on $ \R^{1+n}$.  We decompose $P_{-a,-\lambda,s}$ into three terms
\[P_{-a,-\lambda,s}=P_{-,1}+P_{-,2}+P_{-,3},\]with
\[P_{-,1}=\pd_t^2-\Delta_x-s+s^2t^2+2\lambda st,\quad P_{-,2}=-2(\lambda+st)\pd_t+2\lambda\omega\cdot\nabla_x,\quad P_{-,3}=-a\pd_t+a(\lambda+st).\]
Like in Proposition \ref{p1}, the construction of solutions of the form \eqref{GO2}, is based on a suitable Carleman estimate in negative order sobolev spaces. This estimate is given by the following.

\begin{lem}\label{l8} There exists $\lambda_3^*>0$ such that for $\lambda>\lambda_3^*$,   $s\in(0,\lambda)$ with ${\lambda\over s}$ sufficiently large,  we have
\bel{l8a}\norm{P_{-a,-\lambda,s}v}_{H^{-1}_\lambda(\R^{1+n})}\geq Cs^{1/2}\left(\lambda\norm{v}_{L^2(\R^{1+n})}+\norm{v}_{H^1(\R^{1+n})}\right),\ee
with $\tilde{\Omega}$ a domain containing $\overline{\Omega}$, $\tilde{Q}=(-1,T+1)\times \tilde{\Omega}$ and  $C$ a constant depending only on $T$ and $\tilde{\Omega}$.
\end{lem}
\begin{proof} Without lost of generality we assume that $v$ is real valued. We start by proving the following estimate
\bel{l8b}\norm{P_{-,1}v+P_{-,2}v}_{L^2(\R^{1+n})}\geq Cs^{1/2}\left(\lambda\norm{v}_{L^2(\R^{1+n})}+\norm{v}_{H^1(\R^{1+n})}\right),\ee
Note first that
$$\begin{array}{l}P_{-,1}vP_{-,2}v\\
=-2(\lambda+st)\pd_t^2v\pd_tv+2\pd_t^2v\omega\cdot \nabla_xv+2(\lambda+st)\Delta_xv\pd_tv-2(\Delta_xv)\omega\cdot\nabla_xv\\
+2(-s^3t^3-3\lambda s^2t^2+(s^2-2\lambda^2s)t+s\lambda)\pd_tvv+2(s^2t^2+2\lambda st-s)v\omega\cdot\nabla_xv.\end{array}$$
Therefore, repeating the arguments used in Theorem \ref{c1}, for $\lambda$ sufficiently large we get
\[\begin{aligned}\norm{P_{-,1}v+P_{-,2}v}^2_{L^2(\R^{*}_+\times\R^{n})}&\geq 2\abs{\int_{(0,T+1)\times \tilde{\Omega}}P_{-,1}vP_{-,2}vdxdt}\\
\ &\geq s\lambda^2\int_{(0,T+1)\times \tilde{\Omega}}|v|^2dxdt+s\int_{(0,T+1)\times \tilde{\Omega}}(|\nabla_x v|^2+|\pd_t v|^2)dxdt.\end{aligned}\]
From this estimate we deduce \eqref{l8b}. Combining \eqref{l8b} with arguments similar to Lemma \ref{l1}, we deduce \eqref{l8a}.\end{proof}
In a similar way to Lemma \ref{l1bis}, using the Carleman estimate \eqref{l8a}  we can prove the following.

\begin{lem}\label{l3}Let $\omega\in\mathbb S^{n-1}$  and $a\in W^{1,p}(Q)$. Then, for $\lambda>\lambda_3'$, there  exists a bounded operator $S_{a,\lambda}\in \mathcal B(L^2(Q);H^1(Q))$ such that:
\bel{l3a} P_{a,\omega,\lambda}S_{a,\lambda} F=F,\quad F\in L^2(Q),\ee
\bel{l3b}\lambda\norm{S_{a,\lambda} F}_{L^2(Q)}+\norm{S_{a,\lambda} F}_{H^1(Q)}\leq C\norm{F}_{L^2(Q)},\ee
where $C>0$ is independent of $\lambda$.
\end{lem}

Then we can complete the proof of Proposition \ref{p3} by a fixed point argument.

\section{Uniqueness result}
This section is devoted to the proof of Theorem \ref{thm1}.  From now on we set $q=q_2-q_1$ and $a=a_2-a_1$ on $Q$ and  we assume  that $a=q=0$ on $\R^{1+n}\setminus Q$. For all $\theta\in\mathbb S^{n-1}$ and all  $r>0$, we set
\[\partial\Omega_{+,r,\theta}=\{x\in\partial\Omega:\ \nu(x)\cdot \theta>r\},\quad\partial\Omega_{-,r,\theta}=\{x\in\partial\Omega:\ \nu(x)\cdot \theta\leq r\}\]
and $\Sigma_{\pm,r,\theta}=(0,T)\times \partial\Omega_{\pm,r,\theta}$. Here and in the remaining of this text we always assume, without mentioning it, that $\theta$ and $r$ are chosen in such way that $\partial\Omega_{\pm,r,\pm \theta}$ contain  a non-empty relatively open subset of $\partial\Omega$.
Without lost of generality we assume that there exists $\epsilon>0$ such that  for all $\omega\in\{\theta\in\mathbb S^{n-1}:|\theta-\omega_0|\leq\epsilon\} $ we have $\partial\Omega_{-,\epsilon,\omega}\subset V'$. From now on we will decompose the proof of Theorem \ref{thm1} into two steps. We start by considering the recovery of the damping coefficient by proving that condition \eqref{thm1a} implies $a_1=a_2$. For this purpose we use the tools introduced in Section 4 and the two different geometric optics solutions of Section 5. Next, we consider the recovery of the potential. For this purpose, we prove that,  for $a_1=a_2$, \eqref{thm1a} implies $q_1=q_2$. In that step we will change  the form of the GO solutions by taking into account the fact that $a_1=a_2$.

\subsection{Recovery of the damping coefficient}

In this section we will prove that \eqref{thm1a} implies $a_1=a_2$. For this purpose we will use the notation of Section 4 and the results of the previous  sections. 
Let   $\lambda >\lambda_3$ and fix $\omega\in\{\theta\in\mathbb S^{n-1}:|\theta-\omega_0|\leq\epsilon\} $. According to Proposition \ref{p1}, we can introduce
\[u_1(t,x)=e^{-\lambda(t+x\cdot\omega)}(b_{1,\lambda}(t,x)+w_1(t,x)),\ (t,x) \in Q,\]
where $u_1\in H^1(Q)$ satisfies $\partial_t^2u_1-\Delta_xu_1-a_1\pd_tu_1+q_1u_1=0$, $b_{1,\lambda}$ is given by \eqref{conde1} with $\zeta\cdot(1,-\omega)=0$ and  $w_1$ satisfies \eqref{CGO11}. Moreover, in view of Proposittion \ref{p3}, we consider $u_2\in H^1(Q)$ a solution of $\partial_t^2u_2-\Delta_xu_2+a_2\pd_tu_2+q_2u_2=0$, of the form 
\[u_2(t,x)=e^{\lambda (t+x\cdot\omega)}(b_{2,\lambda}(t,x)+w_2(t,x)),\quad (t,x)\in Q,\]
 with $b_{2,\lambda}$  given by \eqref{conde2},  $w_2$ satisfies \eqref{CGO11}.
In view of Proposition \ref{p6}, there exists a unique weak solution $z_1\in H_\Box (Q)$ of
 \bel{eq3}
\left\{
\begin{array}{ll}
 \partial_t^2z_1-\Delta_xz_1 +a_1\pd_tz_1+q_1z_1=0 &\mbox{in}\ Q,
\\

\tau_{0}z_1=\tau_{0}u_2. &
%\\
%v(\cdot ,1,\cdot )=e^{i\theta}v(0,\cdot,\cdot)

\end{array}
\right.
\ee
Then, $u=z_1-u_2$ solves
  \bel{eq4}
\left\{\begin{array}{ll}
 \partial_t^2u-\Delta_xu +a_1\pd_tu+q_1u=(a_2-a_1)\pd_tu_2+(q_2-q_1)u_2 &\mbox{in}\ Q,
\\
u(0,x)=\partial_tu(0,x)=0 & \mathrm{on}\  \Omega,\\

u=0 &\mbox{on}\ \Sigma.
%\\
%v(\cdot ,1,\cdot )=e^{i\theta}v(0,\cdot,\cdot)
\end{array}\right.
\ee
Since $(a_2-a_1)\pd_tu_2+(q_2-q_1)u_2\in L^2(Q)$, applying    the theory developed in \cite[Chapter 3, Section 8]{LM1}, we can deduce that this problem admits a unique solution $u$ lying in $\mathcal C^1([0,T];L^2(\Omega))\cap \mathcal C([0,T];H^1_0(\Omega))$. Combining this with \cite[Theorem 2.1]{LLT}, we deduce that in addition  $\partial_\nu u\in L^2(\Sigma)$. Therefore, we have $u\in  H^1(Q)\cap H_{\Box}(Q)$ with $\partial_\nu u\in L^2(\Sigma)$. Using the fact that $u_1\in H^1(Q)\cap H_\Box(Q)$, we deduce that $(\partial_tu_1,-\nabla_xu_1)\in H_{\textrm{div}}(Q)=\{F\in L^2(Q;\mathbb C^{n+1}):\ \textrm{div}_{(t,x)}F\in L^2(Q)\}$. Therefore, in view of \cite[Lemma 2.2]{Ka}, we can apply the Green formula to get
\[\int_Q u(\Box u_1)dxdt =-\int_Q(\partial_t u\partial_t u_1-\nabla_x u\cdot\nabla_xu_1)dxdt+\left\langle(\partial_tu_1,-\nabla_xu_1)\cdot \textbf{n}, u\right\rangle_{H^{-{1\over2}}(\partial Q),H^{{1\over2}}(\partial Q)},\]
with $\textbf{n}$ the outward unit normal vector to $\partial Q$. In the same way, we find
\[\int_Q u_1(\Box u)dxdt =-\int_Q(\partial_t u\partial_t u_1-\nabla_x u\cdot\nabla_xu_1)dxdt+\left\langle(\partial_tu,-\nabla_xu)\cdot \textbf{n}, u_1\right\rangle_{H^{-{1\over2}}(\partial Q),H^{{1\over2}}(\partial Q)}.\]
In addition, since $u,u_1\in H^1(Q)$, we have
\[\int_Qa_1\pd_tuu_1dxdt=\int_\Omega a_1(T,x)u(T,x)u_1(T,x)dx-\int_Q\pd_ta_1 uu_1dxdt-\int_Qu a_1 \pd_tu_1dxdt.\]
From these three formulas we deduce that
\[\begin{array}{l}\int_Q(a_2-a_1)\pd_tu_2u_1dxdt+\int_Q(q_2-q_1)u_2u_1dxdt\\
\ \\
=\int_Q u_1(\Box u+a_1\pd_t u+q_1 u)dxdt-\int_Q u(\Box u_1-a_1\pd_tu_1+q_1u_1)dxdt\\
\ \\
=\int_\Omega a_1(T,x)u(T,x)u_1(T,x)dx-\int_Q\pd_ta_1 uu_1dxdt+\left\langle(\partial_tu,-\nabla_xu)\cdot \textbf{n}, u_1\right\rangle_{H^{-{1\over2}}(\partial Q),H^{{1\over2}}(\partial Q)}\\
\ \\
\ \ \ -\left\langle(\partial_tu_1,-\nabla_xu_1)\cdot \textbf{n}, u\right\rangle_{H^{-{1\over2}}(\partial Q),H^{{1\over2}}(\partial Q)}.\end{array}\]
On the other hand we have $u_{|t=0}=\partial_tu_{|t=0}=u_{|\Sigma}=0$ and condition \eqref{thm1a} implies that $u_{|t=T}=\partial_\nu u_{|V}=0$. Combining this with the fact that $u\in \mathcal C^1([0,T];L^2(\Omega))$ and $\partial_\nu u\in L^2(\Sigma)$, we obtain
\begin{equation}\label{t3a} \int_Qa\pd_tu_2u_1dxdt+\int_Qqu_2u_1dxdt=-\int_{\Sigma\setminus V}\partial_\nu uu_1d\sigma(x)dt+\int_\Omega \partial_tu(T,x)u_1(T,x)dx-\int_Q\pd_ta_1 uu_1dxdt.\end{equation}
Applying  the Cauchy-Schwarz inequality to the first expression on the right hand side of this formula and using the fact that  $(\Sigma\setminus V)\subset {\Sigma}_{+,\epsilon,\omega}$, we get
\[\begin{aligned}\abs{\int_{\Sigma\setminus V}\partial_\nu uu_1d\sigma(x)dt}&\leq\int_{{\Sigma}_{+,\epsilon,\omega}}\abs{\partial_\nu ue^{-\lambda(t+\omega\cdot x)}(b_{1,\lambda}+w_1)} d\sigma(x)dt \\
 \ &\leq C(1+\norm{w_1}_{L^2(\Sigma)})\left(\int_{{\Sigma}_{+,\epsilon,\omega}}\abs{e^{-\lambda(t+\omega\cdot x)}\partial_\nu u}^2d\sigma(x)dt\right)^{\frac{1}{2}},\end{aligned}\]
for some $C$ independent of $\lambda$. On the other hand, one can check that
\[\norm{w_1}_{L^2(\Sigma))}\leq C\norm{w_1}_{L^2(Q)}^{{1\over 2}}\norm{w_1}_{H^1(Q)}^{{1\over 2}}.\]
Combining this with \eqref{CGO11}, we obtain
$$\abs{\int_{\Sigma\setminus V}\partial_\nu uu_1d\sigma(x)dt}\leq C \lambda^{{3-2\alpha\over 6}}\left(\int_{{\Sigma}_{+,\epsilon,\omega}}\abs{e^{-\lambda(t+\omega\cdot x)}\partial_\nu u}^2d\sigma(x)dt\right)^{\frac{1}{2}}.$$
 In the same way, we have
\[\begin{aligned}\abs{\int_\Omega \partial_tu(T,x)u_1(T,x)dx}&\leq\int_{\Omega}\abs{\partial_t u(T,x)e^{-\lambda(T+\omega\cdot x)}\left(e^{-i\xi\cdot(T,x)}+w_1(T,x)\right)}dx \\
 \ &\leq C \lambda^{{3-2\alpha\over 6}}\left(\int_{\Omega}\abs{e^{-\lambda(T+\omega\cdot x)}\partial_t u(T,x)}^2dx\right)^{\frac{1}{2}}\end{aligned}\]
and conditions \eqref{condf}, \eqref{CGO11} imply
$$\begin{aligned}\abs{\int_Q\pd_ta_1 uu_1dxdt}&\leq C\norm{a_1}_{W^{1,p}(Q)}(1+\norm{w_1}_{H^1(Q)})\left(\int_{Q}\abs{e^{-\lambda(t+\omega\cdot x)} u(t,x)}^2dxdt\right)^{\frac{1}{2}}\\
\ &\leq C\lambda^{{3-\alpha\over3}}\left(\int_{Q}\abs{e^{-\lambda(t+\omega\cdot x)} u(t,x)}^2dxdt\right)^{\frac{1}{2}}.\end{aligned}$$
Combining these estimates with the Carleman estimate \eqref{c1a} and applying the fact that $u_{|t=T}=\partial_\nu u_{|\Sigma_{-,\omega}}=0$, ${\partial\Omega}_{+,\epsilon,\omega}\subset {\partial\Omega}_{+,\omega}$, we find
\begin{eqnarray}\label{loi}&&\abs{\int_Q(a_2-a_1)\pd_tu_2u_1dxdt+\int_Q(q_2-q_1)u_2u_1dxdt}^2\cr
&&\leq C\lambda^{{3-2\alpha\over 3}}\left(\int_{{\Sigma}_{+,\epsilon,\omega}}\abs{e^{-\lambda(t+\omega\cdot x)}\partial_\nu u}^2d\sigma(x)dt+\int_\Omega \abs{e^{-\lambda(T+\omega\cdot x)}\partial_tu(T,x)}^2dx+\lambda\int_{Q}e^{-2\lambda(t+\omega\cdot x)} \abs{u}^2dxdt\right)\cr
&&\leq \epsilon^{-1}C\lambda^{{3-2\alpha\over 3}}\left(\int_{{\Sigma}_{+,\omega}}\abs{e^{-\lambda(t+\omega\cdot x)}\partial_\nu u}^2\omega\cdot \nu(x)d\sigma(x)dt+\int_\Omega \abs{e^{-\lambda(T+\omega\cdot x)}\partial_tu(T,x)}^2dx\right)\cr
&&\ \ \ +C\lambda^{{6-2\alpha\over 3}}\int_{Q}e^{-2\lambda(t+\omega\cdot x)} \abs{u}^2dxdt\cr
&&\leq \epsilon^{-1}C\lambda^{-{2\alpha\over 3}}\left(\int_Q\abs{ e^{-\lambda(t+\omega\cdot x)}(\partial_t^2-\Delta_x+a_1\pd_t +q_1)u}^2dxdt\right)\cr
&&\leq \epsilon^{-1}C\lambda^{-{2\alpha\over 3}}\left(\int_Q\abs{ e^{-\lambda(t+\omega\cdot x)}(a\pd_tu_2+qu_2)}^2dxdt\right)
.\end{eqnarray}
Here $C>0$ stands for some generic constant independent of $\lambda$. In view of \eqref{condf} and \eqref{CGO11}, we have
$$\int_Q\abs{ e^{-\lambda(t+\omega\cdot x)}(a\pd_tu_2+qu_2)}^2dxdt\leq C\lambda^2$$
and we deduce  that
\[\abs{\int_Q(a_2-a_1)\pd_tu_2u_1dxdt+\int_Q(q_2-q_1)u_2u_1dxdt}\leq C\lambda^{{3-\alpha \over3}}.\]
It follows 
\begin{equation}\label{t3cc}\lim_{\lambda\to+\infty}{\int_Q(a_2-a_1)\pd_tu_2u_1dxdt+\int_Q(q_2-q_1)u_2u_1dxdt\over\lambda}=0.\end{equation}
On the other hand, from the properties of $u_1$, $u_2$, one can check that
\[{\int_Q(a_2-a_1)\pd_tu_2u_1dxdt+\int_Q(q_2-q_1)u_2u_1dxdt\over\lambda}=\int_{\R^{1+n}}a(t,x)b_{1,\lambda}(t,x)b_{2,\lambda}(t,x)dxdt+ \int_{Q}Z(t,x) dxdt,\]
where $ Z$ satisfies
\[\abs{\int_{Q}Z(t,x)  dxdt}\leq C\lambda^{-\frac{\alpha}{3}},\]
with $C$ independent of $\lambda$. Combining this with \eqref{t3cc}, we deduce that 
\bel{t3r}\lim_{\lambda\to+\infty}\int_{\R^{1+n}}a(t,x)b_{1,\lambda}(t,x)b_{2,\lambda}(t,x)dxdt=0.\ee
But, in view of \eqref{a1a} and \eqref{a1b}, we have
\[\norm{a-a_\lambda}_{L^\infty(\R^{1+n})}\leq \norm{\tilde{a}_1-a_{1,\lambda}}_{L^\infty(\R^{1+n})}+\norm{\tilde{a}_2-a_{2,\lambda}}_{L^\infty(\R^{1+n})}\leq C\lambda^{-{\alpha \over 3}}.\]
Then, using \eqref{condf} and the fact that supp$a_\lambda\cup$supp$a\subset B_{R+1}$, in light of \eqref{t3r}, we obtain
\bel{t3s}\lim_{\lambda\to+\infty}\int_{\R^{1+n}}a_\lambda(t,x)b_{1,\lambda}(t,x)b_{2,\lambda}(t,x)dxdt=\lim_{\lambda\to+\infty}\int_{\R^{1+n}}a(t,x)b_{1,\lambda}(t,x)b_{2,\lambda}(t,x)dxdt=0.\ee
Moreover, in view of \eqref{conde1}-\eqref{conde2}, we find
$$\begin{array}{l}\int_{\R^{1+n}}a_\lambda(t,x)b_{1,\lambda}(t,x)b_{2,\lambda}(t,x)dxdt\\
\ \\
=\int_{\R^{1+n}}a_{\lambda}(t,x)\exp\left({\int_0^{+\infty} a_{\lambda}((t,x)+s(1,-\omega))ds\over 2}\right)e^{-i\zeta\cdot(t,x)}dxdt.\end{array}$$
Decomposing $\R^{1+n}$ into the direct sum $\R^{1+n}=\R(1,-\omega)\oplus(1,-\omega)^{\bot}$ and applying the Fubini's theorem we get
\bel{t3t}\begin{array}{l}\int_{\R^{1+n}}a_\lambda(t,x)b_{1,\lambda}(t,x)b_{2,\lambda}(t,x)dxdt\\
\ \\
=\int_{(1,-\omega)^{\bot}}\left(\int_\R a_{\lambda}(\kappa+\tau(1,-\omega))\exp\left({\int_\tau^{+\infty} a_{\lambda}(\kappa+s(1,-\omega))ds\over 2}\right)\sqrt{2}d\tau\right)e^{-i\zeta\cdot \kappa}d\kappa.\end{array}\ee
In addition, for all $\kappa\in (1,-\omega)^{\bot}$ we have
$$\begin{aligned}\int_\R a_{\lambda}(\kappa+\tau(1,-\omega))\exp\left({\int_\tau^{+\infty} a_{\lambda}(\kappa+s(1,-\omega))ds\over 2}\right)d\tau&=-2\int_\R \pd_\tau\exp\left({\int_\tau^{+\infty} a_{\lambda}(\kappa+s(1,-\omega))ds\over 2}\right)d\tau\\
\ &=-2\left(1-\exp\left({\int_\R a_{\lambda}(\kappa+s(1,-\omega))ds\over 2}\right)\right).\end{aligned}$$
Combining this with \eqref{t3t}, we find
\bel{tete}\begin{array}{l}\int_{\R^{1+n}}a_\lambda(t,x)b_{1,\lambda}(t,x)b_{2,\lambda}(t,x)dxdt\\
\ \\
=-2\sqrt{2}\int_{(1,-\omega)^{\bot}}\left(1-\exp\left({\int_\R a_{\lambda}(\kappa+s(1,-\omega))ds\over 2}\right)\right)e^{-i\zeta\cdot\kappa}d\kappa.\end{array}\ee
Now let us introduce the Fourier transform $\mathcal F_{(1,-\omega)^{\bot}}$ on $(1,-\omega)^{\bot}$ defined by
\[\mathcal F_{(1,-\omega)^{\bot}}f(\xi)=(2\pi)^{-{n\over2}}\int_{(1,-\omega)^{\bot}}f(\kappa)e^{-i\kappa\cdot\xi}d\kappa,\quad f\in L^1((1,-\omega)^{\bot}), \ \ \xi\in(1,-\omega)^{\bot}.\]
 Fixing 
$$f:(1,-\omega)^{\bot}\ni \kappa\mapsto\left(1-\exp\left({\int_\R a_{\lambda}(\kappa+s(1,-\omega))ds\over 2}\right)\right)$$
and applying \eqref{t3s}, \eqref{tete}, we find $\mathcal F_{(1,-\omega)^{\bot}}f=0$. Thus, we have
\[1-\exp\left({\int_\R a_{\lambda}(\kappa+s(1,-\omega))ds\over 2}\right)=0,\quad \kappa\in(1,-\omega)^{\bot}.\]
From this formula and the fact that $a$ is real valued, we deduce that
\[\int_\R a_{\lambda}(\kappa+s(1,-\omega))ds=0,\quad \kappa\in(1,-\omega)^{\bot}.\]
Therefore, taking the Fourier transform on $(1,-\omega)^{\bot}$  of the function
\[(1,-\omega)^{\bot}\ni \kappa\mapsto\int_\R a_{\lambda}(\kappa+s(1,-\omega))ds\]
at $\zeta\in(1,-\omega)^{\bot}$ and applying the Fubini's theorem, we obtain
\[\mathcal F(a_\lambda)(\zeta)=(2\pi)^{-{n+1\over2}}\int_{\R^{1+n}} a_{\lambda}(t,x)e^{-i\zeta\cdot (t,x)}dxdt=0.\]
Then, using the fact that $a_{\lambda}$ converge to $a$ in $L^1(\R^{1+n})$ we obtain  that
for all $\omega\in\{y\in\mathbb S^{n-1}:|y-\omega_0|\leq\epsilon\} $ and  all $\zeta\in(1,-\omega)^\bot$, the Fourier transform $\mathcal F(a)$ of $a$ satisfies
\[\mathcal F(a)(\zeta)=(2\pi)^{-{n+1\over2}}\int_{\R^{1+n}}a(t,x)e^{-i\zeta\cdot(t,x)}dxdt=(2\pi)^{-{n+1\over2}}\lim_{\lambda\to+\infty}\int_{\R^{1+n}} a_{\lambda}(t,x)e^{-i\zeta\cdot (t,x)}dxdt=0.\]
On the other hand, since $a\in L^\infty(\R^{1+n})$ is  supported on $\overline{Q}$ which is compact,  $\mathcal F(a)$ is analytic and it follows that $a=0$ and $a_1=a_2$. This completes the first step of the proof of Theorem \ref{thm1}. We can now consider the recovery of the potential by assuming $a_1=a_2$.

\subsection{Recovery of the potential}

In this subsection we assume that $a_1=a_2$ and  we will prove that \eqref{thm1a} implies $q_1=q_2$. Using the notation of Section 4,  for $\zeta\in (1,-\omega)^\bot$, $\zeta\neq0$, we fix 
\bel{aa1c} b_{1,\lambda}(t,x)=e^{-i\zeta\cdot(t,x)}\exp\left(-{\int_0^{+\infty} a_{1,\lambda}((t,x)+s(1,-\omega))ds\over 2}\right),\ee
\bel{aa1d}b_{2,\lambda}(t,x)=\exp\left({\int_0^{+\infty} a_{1,\lambda}((t,x)+s(1,-\omega))ds\over 2}\right).\ee
Repeating our previous arguments, from \eqref{CGO11} and  \eqref{loi} we deduce
$$\abs{\int_Q(q_2-q_1)u_2u_1dxdt}^2\leq C\lambda^{-{2\alpha\over 3}}\left(\int_Q\abs{ e^{-\lambda(t+\omega\cdot x)}qu_2}^2dxdt\right)\leq C\lambda^{-{2\alpha\over 3}}.$$
Therefore, we have
\bel{aa1e}\lim_{\lambda\to+\infty}\int_Q(q_2-q_1)u_2u_1dxdt=0.\ee
Moreover, estimates \eqref{CGO11} imply
$$\int_Q(q_2-q_1)u_2u_1dxdt=\int_{\R^{1+n}}q(t,x)e^{-i\zeta\cdot(t,x)}dxdt+\int_{Q}W(t,x)dxdt,$$
with
\[\int_{Q}|W(t,x)|dxdt\leq C\lambda^{-{\alpha\over3}}.\]
Combining this with \eqref{aa1e}, for all $\omega\in\{y\in\mathbb S^{n-1}:|y-\omega_0|\leq\epsilon\} $ and  all $\zeta\in(1,-\omega)^\bot$, the Fourier transform $\mathcal F(q)$ of $q$ satisfies
\[\mathcal F(q)(\zeta)=(2\pi)^{-{n+1\over2}}\int_{\R^{1+n}}q(t,x)e^{-i\zeta\cdot(t,x)}dxdt=0.\]
On the other hand, since $q\in L^\infty(\R^{1+n})$ is  supported on $\overline{Q}$ which is compact,  $\mathcal F(q)$ is analytic and it follows that $q=0$ and $q_1=q_2$. This completes the  proof of Theorem \ref{thm1}.

\end{document}